%% file: Mazur.tex
\documentclass{gthack}   
\usepackage[margin=1.4in, left=1.35in, right=1.1in]{geometry}

\usepackage{graphicx}
\usepackage{enumitem}
\usepackage{tikz-cd}

\usepackage{diagrams}
\title[Exotic Mazur manifolds and knot trace invariants]{Exotic Mazur manifolds and knot trace invariants}

\author[Hayden]{Kyle Hayden} \address{Columbia University, New York, NY 10027} 
\email{hayden@math.columbia.edu}

\author[Mark]{Thomas E. Mark} 
\address{University of Virginia, Charlottesville, VA 22904}
\email{tmark@virginia.edu}

\author[Piccirillo]{Lisa Piccirillo} 
\address{Brandeis University, Waltham, MA 02453}
\email{lpiccirillo@math.utexas.edu}

\keyword{}

\theoremstyle{plain}
\newtheorem{thm}{Theorem}[section]    
\newtheorem{lem}[thm]{Lemma}          
\newtheorem{prop}[thm]{Proposition}
\newtheorem{cor}[thm]{Corollary}

\newtheorem{conj}[thm]{Conjecture}
\theoremstyle{definition}
\newtheorem{defn}[thm]{Definition}    

\newtheorem{rem}[thm]{Remark}             
\newcommand{\lk}{\mathrm{lk}}
\newcommand{\tb}{\mathrm{tb}}
\newcommand{\rot}{\mathrm{rot}}
\renewcommand{\S}{\textsection}
\newcommand{\nd}{N}
\newcommand{\gsh}{g_{sh}}
\newcommand{\isoeq}{=}

\newcommand{\s}{\mathfrak{s}}

\newcommand{\spinc}{{\mbox{spin$^c$} }}

\newcommand{\OnS}{{Ozsv\'ath and Szab\'o}}
\newcommand{\zee}{\mathbb{Z}}

\newcommand{\cee}{\mathbb{C}}
\newcommand{\eff}{\mathbb{F}}
\newcommand{\X}{\mathbb{X}}
\newcommand{\K}{\mathcal{K}}

\newcommand{\hfhat}{\widehat{HF}}
\newcommand{\cfhat}{\widehat{CF}}

\newcommand{\ts}{\textstyle}

\newcommand{\cpbar}{\overline{\cee {P}^2}}

\titleformat{\subsubsection}[runin]{\normalfont\bfseries}{\thesubsubsection}{0.75em}{}[.]

\begin{document}

\begin{abstract}
From a handlebody-theoretic perspective, the simplest compact, contractible 4-manifolds, other than the 4-ball, are Mazur manifolds. We produce the first pairs of Mazur manifolds that are homeomorphic but not diffeomorphic.  Our diffeomorphism obstruction comes from our proof that the knot Floer homology concordance invariant $\nu$ is an invariant of the smooth 4-manifold associated to a knot $K \subset S^3$ by attaching an $n$-framed 2-handle to $B^4$ along $K$. We also show (modulo forthcoming work of Ozsv\'ath and Szab\'o) that the concordance invariants $\tau$ and $\epsilon$ are \emph{not} invariants of such 4-manifolds. As a corollary to the existence of exotic Mazur manifolds, we produce integer homology 3-spheres admitting two distinct $S^1\times S^2$ surgeries, resolving a question from Problem 1.16 in Kirby's list \cite{kirby}.
\end{abstract}

\maketitle

\vspace{-0.65cm}

\section{Introduction}

A primary goal of 4-manifold topology is the detection of exotic smooth structures on the simplest 4-manifolds. A pair of smooth manifolds are said to be exotic if they are homeomorphic but not diffeomorphic. Freedman's exotic structures on $\mathbb{R}^4$ provided the first examples of exotic contractible 4-manifolds \cite{freedman}. Exotic structures on contractible 4-manifolds are elusive in part because most tractable smooth 4-manifold invariants are trivial on manifolds with vanishing second homology. Our goal in this paper is to use tools from knot theory and 3-manifold topology to detect exotic smooth structures on simple 4-manifolds.

From the perspective of handlebody theory, the simplest contractible 4-manifolds, other than $B^4$, are Mazur manifolds. Introduced in \cite{mazur}, a \emph{Mazur manifold} is any contractible 4-manifold obtained by attaching a single 1- and 2-handle to $B^4$. In \cite{akbulut:cork}, Akbulut showed that one of Mazur's original examples $W$ admits a \emph{relatively} exotic smooth structure; that is, there is a self-diffeomorphism of $\partial W$ that extends to a homeomorphism of $W$ but not a diffeomorphism of $W$.  In \cite{akbulut-ruberman}, Akbulut and Ruberman produced the first examples of exotic compact, contractible 4-manifolds. However, their construction naturally yields examples with complicated handlebody structures; it remained open whether Mazur manifolds admit exotic smooth structures. 

\begin{thm}\label{thm:mazur}
There exist infinitely many pairs of exotic Mazur manifolds.
\end{thm}

\begin{figure}\center
\def\svgwidth{.8\linewidth}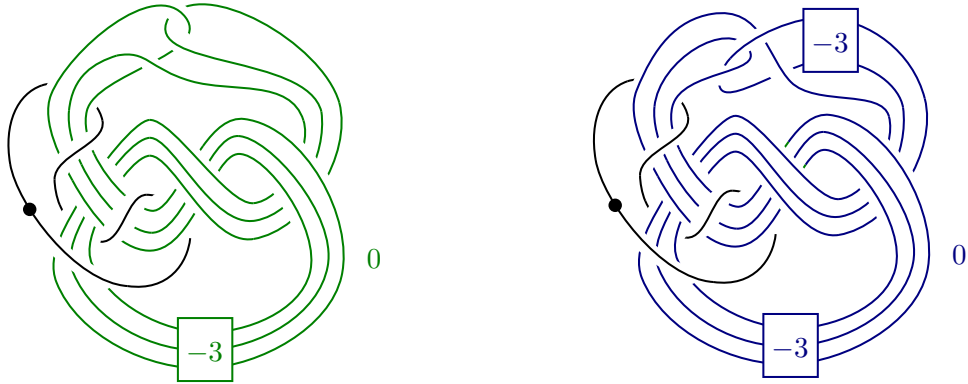 
\caption{An exotic pair of Mazur manifolds.}
\label{fig:example1}
\end{figure}

Our exotic Mazur manifolds also resolve a question about Dehn surgery.  Gabai \cite{gabai:foliations} famously proved that the only knot in $S^3$ admitting an $S^1\times S^2$ surgery is the unknot, proving Property R and answering the main question from Problem 1.16 in Kirby's problem list \cite{kirby}. However, Problem 1.16 also raises the analagous question for other integer homology spheres:

\begin{conj}[\cite{kirby-melvin,kirby}]
\label{conj:propR}
Let $Y$ be an integer homology sphere containing a knot $K$ which admits a surgery to $S^1 \times S^2$. Then $K$ is unique up to ambient diffeomorphism.
\end{conj}

An integer homology sphere $Y$ contains a knot $K$ with a surgery to $S^1 \times S^2$ if and only if $Y$ bounds a Mazur manifold; in this case $K$ is the belt sphere of the 2-handle. Each pair of exotic Mazur manifolds from Theorem~\ref{thm:mazur} have the same boundary $Y$, and thus $Y$ comes equipped with two such belt spheres. It is not hard to show that the inequivalence of the smooth structures implies that these belt spheres are inequivalent, hence yield counterexamples to Conjecture~\ref{conj:propR}; see \S\ref{subsec:propR}.

\begin{cor}\label{cor:propR}
There exist infinitely many irreducible integer homology spheres each of which contains a pair of distinct knots along which zero-surgery yields $S^1 \times S^2$.
\end{cor}

As mentioned above, there are no known smooth 4-manifold invariants suited to directly distinguishing exotic pairs of compact, contractible 4-manifolds. Instead, to our pairs of homeomorphic Mazur manifolds $W$ and $W'$ we associate knot traces $X$ and $X'$. A \emph{knot trace} $X_n(K)$ is a 4-manifold obtained by attaching an $n$-framed 2-handle to $B^4$ along $K$. The main idea in our proof of Theorem~\ref{thm:mazur} is to use 3-dimensional techniques to show that that if the Mazur manifolds $W$ and $W'$ are diffeomorphic then the associated traces $X$ and $X'$ are diffeomorphic; see \S\ref{subsec:associatedtraces}. 

It then suffices to distinguish the smooth structures on these knot traces. 
Exotic knot traces were first produced by Akbulut in \cite{akbulut:exotic}, and further examples appear in \cite{kalmar-stipsicz,yasui:conc}. By repurposing the concordance invariant $\nu$ from knot Floer homology \cite{oz-sz:rational}, we develop a new integer-valued invariant of smooth, oriented knot traces.

\begin{thm}\label{thm:nu}
If the oriented knot traces $X_n(K)$ and $X_n(K')$ are diffeomorphic, then $\nu(K)=\nu(K')$ except possibly if $n <0$ and $\{\nu(K), \nu(K')\}=\{0,1\}$.
\end{thm}

We remark that we are not aware of any instances realizing the exception in the theorem. In fact, we show in Proposition \ref{prop:nudual} that for all pairs of knots $K, K'$ in the literature with $X_n(K)\cong X_n(K')$, regardless of $n$ or values of $\nu$, we have  $\nu(K)=\nu(K')$. 

The relative computability of $\nu$ makes it an especially effective trace invariant. In the proof of Theorem~\ref{thm:mazur} (and indeed in all of our examples), direct calculations are made possible using three properties of $\nu$: (1) its concordance invariance \cite{oz-sz:rational}; (2) a twist inequality (Proposition~\ref{prop:nutwist}); and (3) a formula of Levine \cite{levine:PL} which determines the behavior of $\nu$ under certain satellite operations (Corollary~\ref{cor:nu-formula}). We use $\nu$ to detect new types of exotic knot traces, including the first pair of hyperbolic knots with exotic zero-traces (in \S\ref{sec:hyperbolic}).  

The invariant $\nu$ is closely related to the Heegaard Floer concordance invariant $\tau$ \cite{oz-sz:genus,rasmussenthesis}, and it has been asked whether $\tau$ is a zero-trace invariant, e.g.~\cite[Problem 12]{bonn-questions}. Assuming a forthcoming result of Ozsv\'ath and Szab\'o, we show that the answer is \emph{no}. In particular,  Ozsv\'ath and Szab\'o defined knot invariants $\underline{\tau}$, $\underline{\epsilon}$, and $\underline{\nu}$ using bordered algebras \cite{oz-sz:matchings}, and they have announced a proof  (to appear in \cite{oz-sz:matchingsHFK}) that these invariants agree with their well-known analogs $\tau$, $\epsilon$, and $\nu$. These bordered invariants have the advantage of being highly computable; see \cite{hfk-calc}. In \S\ref{subsec:other}, we produce pairs of knots $K$ and $K'$ with diffeomorphic zero-traces such that $\underline{\tau}(K)\neq \underline{\tau}(K')$ and $\underline{\epsilon}(K) \neq \underline{\epsilon}(K')$. This shows:

\begin{thm}\label{thm:tau}
The concordance invariants $\tau$ and $\epsilon$ are not zero-trace invariants.
\end{thm}

Previous work on exotic knot traces has relied heavily on the genus function of $X_n(K)$; recall that the \emph{$n$-shake genus} $\gsh^n(K)$ of $K\subset S^3$ is defined to be the minimal genus of a smoothly embedded surface generating the second homology of $X_n(K)$. In fact, to the authors' knowledge, all previously known examples of exotic knot traces can be distinguished via their genus functions using an adjunction inequality \cite{akbulut:exotic,akbulut-matveyev,kalmar-stipsicz,yasui:conc}. In light of this interest, we point out that $\nu$ provides lower bounds on the shake genus for a range of framings.

\begin{thm}\label{thm:nushakeboundclean} If $K$ is a knot in $S^3$ and $n \in \zee$ satisfies $n=0$ or $2-2\nu(K)<n\leq 2\nu(K) -2$, then $\nu(K) \leq \gsh^n(K).$
\end{thm}

A sharper, albeit more complicated, statement appears in Proposition \ref{nushakebound}.  We also observe that Theorem \ref{thm:nushakeboundclean} gives an extension of  the shake-slice Bennequin inequality for Legendrian knots in $S^3$, cf.~Corollary \ref{cor:shsliceben}.

However, we emphasize that $\nu$ can directly distinguish between exotic smooth structures on knot traces $X$ and $X'$ even when the genus functions of $X$ and $X'$ are unknown. We illustrate this difference in \S\ref{subsubsec:traces}, where we use $\nu$ to distinguish pairs of homeomorphic knot traces that are not readily distinguished by their genus functions.

This paper is organized as follows: In \S\ref{sec:exmaz}, we construct homeomorphic Mazur manifolds and prove Theorem \ref{thm:mazur}. In \S\ref{subsec:propR}, we prove Corollary~\ref{cor:propR}. In \S\ref{sec:HF}, we recall the relevant background on knot Floer homology and prove Theorems \ref{thm:nu} and \ref{thm:tau}. The casual reader can consider the paper complete after \S\ref{sec:HF}. The remainder of the paper serves to develop some tools and examples related to our main work which may be of future use.  In \S\ref{sec:bounds}, we prove Theorem \ref{thm:nushakeboundclean}. In \S\ref{sec:examples}, we describe modifications of our main constructions that yield exotica with additional interesting properties: namely exotic Mazur pairs admitting Stein structures, exotic Mazur pairs whose handlebody diagrams have unknotted 2-handles, exotic Mazur pairs with hyperbolic boundary and exotic zero-traces with hyperbolic boundary, and exotic zero-traces with unknown genus functions.

For additional documentation regarding our computer calculations, see \cite{exotic-data}.

\smallskip
\emph{Acknowledgements.} \ 
This project originated at the conference \emph{Geometric Structures on 3- and 4-Manifolds} at IUC Dubrovnik in summer 2018; we thank the organizers for providing such a stimulating environment.  Frank Swenton's Kirby calculator \cite{KLO} was indispensable at various points of this project.  This work benefited greatly from several insightful conversations with Danny Ruberman.  L.P. is indebted to her adviser, John Luecke, for his guidance and expertise throughout this project. In winter 2019 Min Hoon Kim and JungHwan Park informed the authors that they have an independent proof of Corollary \ref{cor:propR}; see \cite{kim-park:propR}. 

K.H.~was supported in part by NSF grant DMS-1803584. T.M. was supported in part by a grant from the Simons Foundation (523795, TM). L.P. was supported in part by NSF grant DMS-1902735.

\section{Exotic Mazur manifolds}\label{sec:exmaz}

We begin by constructing pairs of homeomorphic Mazur manifolds in \S\ref{subsec:construction}. In \S\ref{subsec:associatedtraces}, we construct associated pairs of knot traces and show that any diffeomorphism between certain homeomorphic Mazur manifolds induces a diffeomorphism between their associated traces. In  \S\ref{subsec:obstruction}, we use the trace invariance of $\nu$ (Theorem \ref{thm:nu}) to show that certain pairs of associated traces are not diffeomorphic, hence the underlying  Mazur pairs are exotic. 

\subsection{Constructing homeomorphic Mazur manifolds}\label{subsec:construction}

Our construction will rely on the satellite operation with the patterns  $P$ and $Q$ in the solid torus $S^1 \times D^2$ shown in Figure~\ref{fig:patterns}. The pattern $P$ is the well-known Mazur pattern, and the pattern $Q$ was defined by Yasui \cite{yasui:conc}. In all figures, a boxed integer indicates positive full twists. For any pattern $R$, we define $R_n$ to be the pattern obtained by adding $n$ positive full twists to $R$ along a disk $\{pt\} \times D^2 \subset S^1 \times D^2$ that meets $R$ transversely.

\begin{figure}\center
\def\svgwidth{.9\linewidth}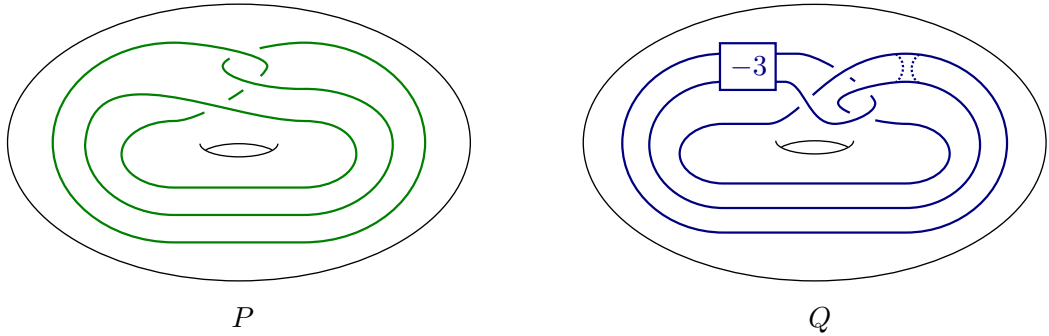 
\caption{The Mazur pattern $P$ and Yasui's concordance pattern $Q$.}\label{fig:patterns}
\end{figure}

\begin{rem} 
For all integers $n$, the pattern $Q_n$ is concordant in $S^1\times D^2\times [0,1]$ to the pattern $S^1\times \{pt\}$. This can be seen using the saddle move depicted in Figure \ref{fig:patterns}.
\end{rem}

Let $Z$ be a smooth 4-manifold and $K\subset \partial Z$ be a framed knot. We define knots $P(K)$ and $Q(K)$ in $\partial Z$ by taking the satellite of $K$ with patterns $P$ and $Q$ respectively, and we define 4-manifolds $Z_P$ and $Z_Q$ by attaching 2-handles to $Z$ along $P(K)$ and $Q(K)$, respectively. The 2-handles are framed so that a pushoff of $P(K)$ is homologous in $ V\smallsetminus \mathring{N}(P)$ to $S^1\times\{pt\}\subset\partial V$, and similarly for $Q(K)$.

\begin{prop}\label{prop:homeo}
$Z_P$ and $Z_Q$ are homeomorphic.
\end{prop}

\begin{rem}
A similar statement and proof appears in \cite{yasui:conc}. Both arguments rely on a handle calculus trick originating in \cite{akbulut2dclasses} and \cite{akbulut:exotic}. 
\end{rem}

\begin{proof}
Figure \ref{fig:homeo}(a) is a schematic handle diagram representing the handle attached to $Z$ to yield $Z_P$, where the framing on the knot $K$ is taken to be $n$, and (h) is the corresponding diagram for $Z_Q$. The intervening sequence of diagrams shows that $Z_P$ and $Z_Q$ are related by twisting along an embedded Mazur cork. Since cork twisting amounts to cutting out and regluing a contractible submanifold, it  preserves homeomorphism type by Freedman \cite{freedman}. 
\end{proof}

\begin{figure}\center
\def\svgwidth{.8\linewidth}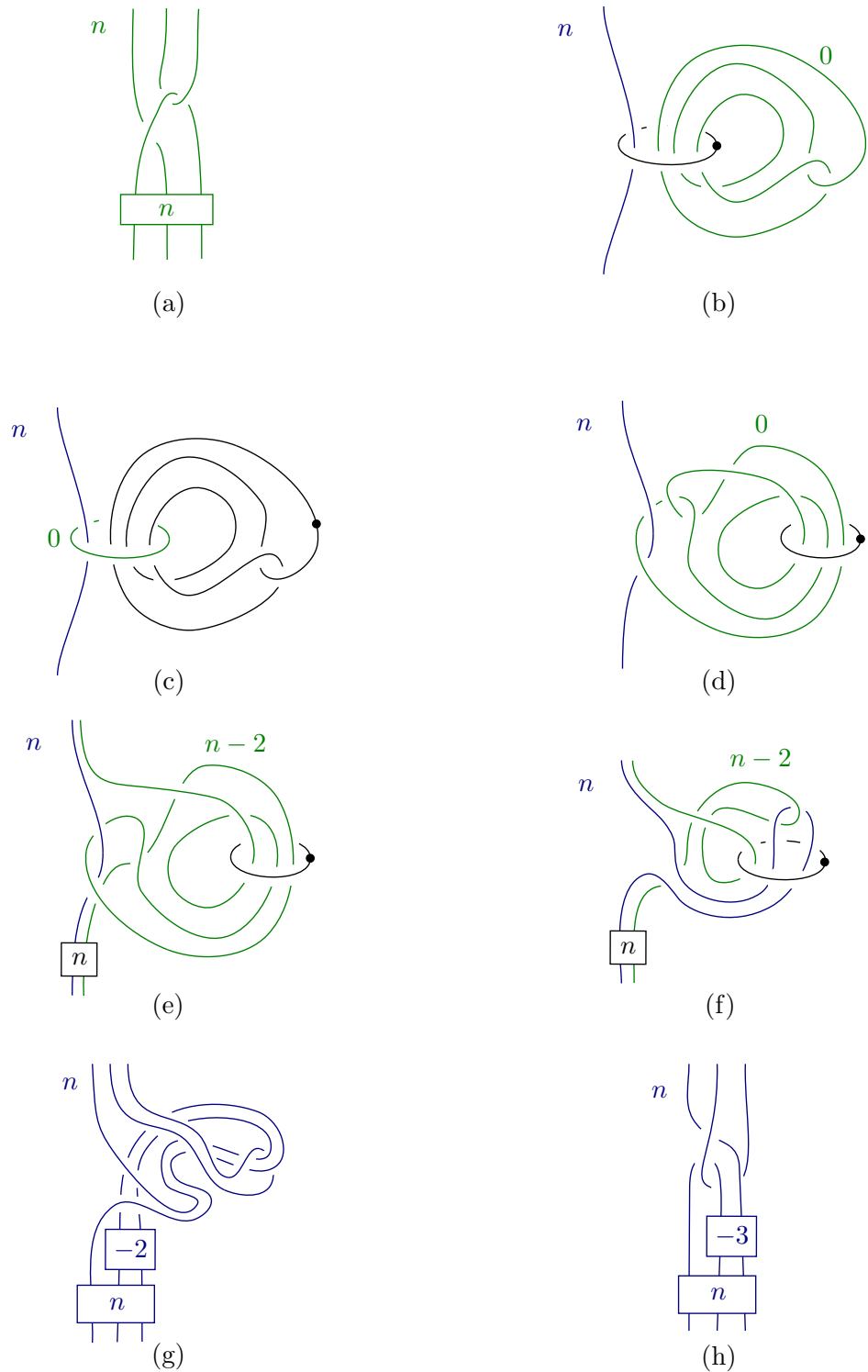 
\caption{Handle calculus relating two homeomorphic 4-manifolds differing by a single 2-handle attachment. The transition from (a) to (b) consists of a 1-/2-pair creation and handleslides; (b) to (d) is a cork twist and isotopy; (d) to (f) is a handleslide and isotopy; (f) to (h) is a pair of handleslides, followed by a 1-/2-pair cancellation and isotopy.}
\label{fig:homeo}
\end{figure}

\begin{figure}\center
\def\svgwidth{.35\linewidth}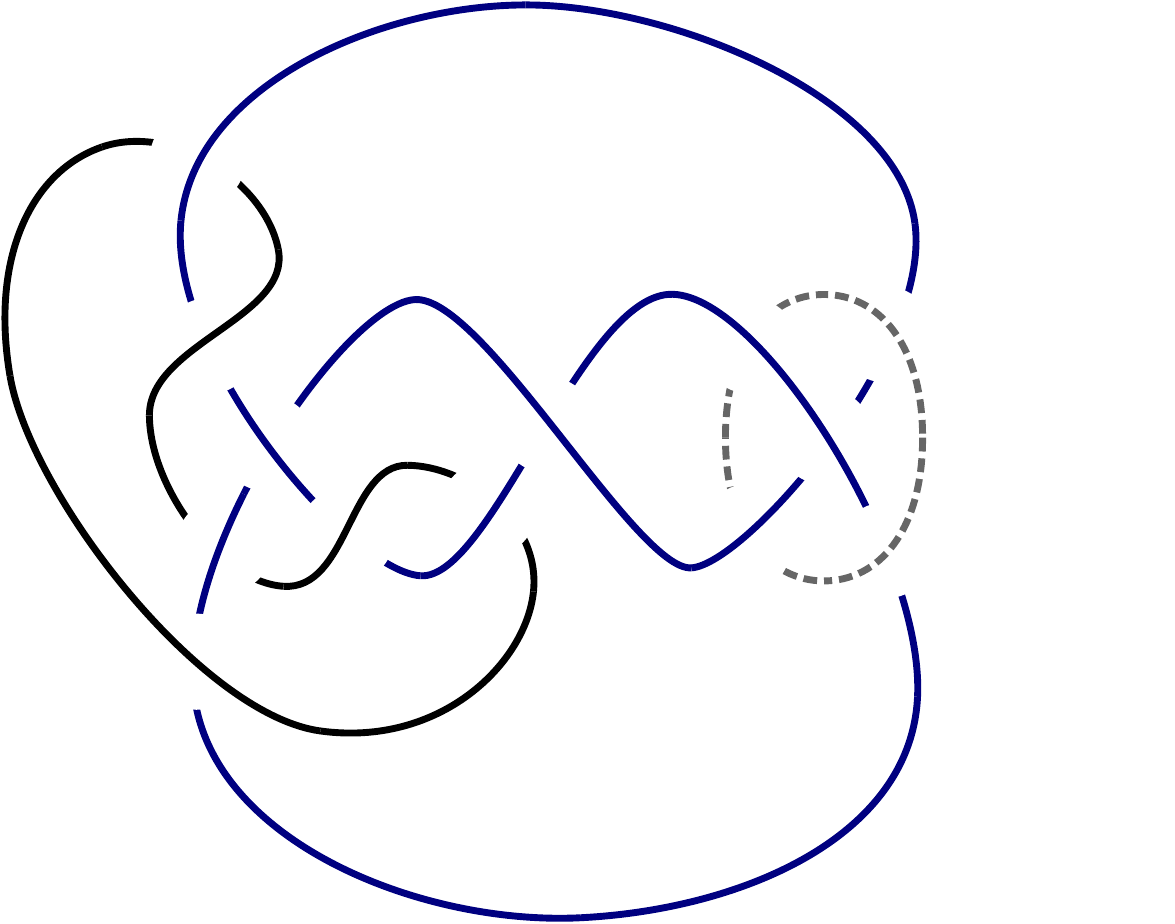 
\caption{An example link $K\cup C$, affectionately known as the trefoil on the phone}\label{fig:tref-phone}
\end{figure}

\begin{defn}
For any knot $J$ in $V=S^1\times D^2$ with winding number one, define the \emph{zero-framing} of $J$ to be unique framing curve homologous to $S^1\times \{pt\} \subset S^1 \times \partial D^2$.
Define $V_0(J)$ to be the 3-manifold obtained from $V$ by zero-framed Dehn surgery along $J$. 
\end{defn}

\begin{rem}
The proof of Proposition \ref{prop:homeo} also shows that $V_0(P)\cong V_0(Q)$. Indeed, we can consider Figure \ref{fig:homeo} to be a sequence of handle moves taking place in $V$, specializing to $n = 0$.
\end{rem}

By definition, a Mazur manifold is any contractible 4-manifold built with a single 0-, 1-, and 2-handle. It is straightforward to check that a manifold with such a handle decomposition is contractible if and only if the attaching curve of the 2-handle passes algebraically once over the 1-handle. Our construction for Theorem~\ref{thm:mazur} begins with a 2-component link $L=K \cup C$ in $S^3$, where $C$ is an unknot and $\lk(K,C)=\pm 1$; for an example see Figure~\ref{fig:tref-phone}. Define a pair of Mazur manifolds $W_{L,n}$ and $W'_{L,n}$ by viewing $C$ as a dotted circle and attaching an $n$-framed 2-handle along $P_n(K)$ and $Q_n(K)$, respectively.  When $n=0$ we will just denote these manifolds $W_L$ and $W'_L$. See Figure~\ref{fig:example1} for an example of $W_L$ and $W'_L$ when $K\cup C$ is the trefoil on the phone from Figure~\ref{fig:tref-phone}. Note that we can view $K$ as a curve in the boundary of the $S^1\times B^3$ given by the (dotted) 1-handle $C$. Thus applying Proposition~\ref{prop:homeo} to $Z  = S^1\times B^3$ yields:

\begin{cor} Any such Mazur manifolds $W_{L,n}$ and $W'_{L,n}$ are homeomorphic.
\end{cor}

\subsection{Associated knot traces} \label{subsec:associatedtraces}

For the rest of this section, $W_{L,n}$ and $W'_{L,n}$ will be Mazur manifolds that arise from a link $L=K \cup C$ using our construction in \S\ref{subsec:construction}. Moreover, we will take as an additional hypothesis that the exterior of $K$ in $S^3_0(C)$, denoted $E$, is not homeomorphic to the solid torus $V=S^1 \times D^2$, nor does its JSJ decomposition \cite{jaco-shalen,johannson} contain a component homeomorphic to $V_0(P)$. For example, it suffices that the image of $K$ in $S^1\times S^2 \cong S^3_0(C)$ is a hyperbolic knot whose complement has volume different from that of $V_0(P)$.

\begin{thm}\label{thm:reduce}
Under the hypotheses above, if $W_{L,n} \cong W_{L,n}'$ then $X_n(P_n(K)) \cong X_n(Q_n(K))$.
\end{thm}

\begin{proof}
Suppose there exists a diffeomorphism $\Phi:W_{L,n}\to W'_{L,n}$, and let $\phi$ denote the induced  homeomorphism between $Y=\partial W_{L,n}$ and $Y'=\partial W'_{L,n}$. Let $\gamma\subset Y$ be the knot given by the meridian of the dotted circle in our handle description of $W_{L,n}$.  We can build $X_n(P_n(K))$ by attaching a 0-framed 2-handle to $W_{L,n}$ along $\gamma$. Then $\Phi$ extends to a diffeomorphism from $X_n(P_n(K))$ to $X'$, where $X'$ is the manifold obtained by attaching  a 2-handle to $W'_{L,n}$ along the framed curve $\phi(\gamma)$, which is the framed image of the 0-framed curve $\gamma$. See Figure~\ref{fig:reduce}. In the induced handlebody diagram for $X'$, the framing on $\phi(\gamma)$ is some integer $m \in \zee$. To prove the theorem it suffices to prove $X'\cong X_n(Q_n(K))$.

\begin{figure}\center
\def\svgwidth{\linewidth}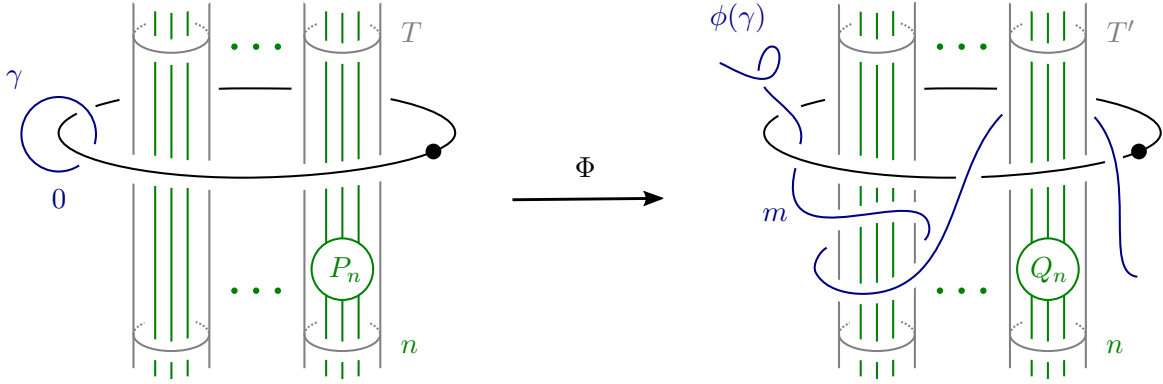 
\caption{Schematic handle diagrams for $X_n(P_n(K))$ and $X'$.}\label{fig:reduce}
\end{figure} 

To understand $X'$, we study the curve $\phi(\gamma)\subset Y'$. Observe that $Y$ decomposes as a union of $E=S^3_0(C) \setminus \mathring{N}(K)$ and $V_0(P)$ along a torus $T$, with $Y'$ similarly decomposing into $E$ and $V_0(Q)$ along a torus $T'$. As verified in Snappy and Sage \cite{snappy,sagemath}, $V_0(P) \cong V_0(Q) $ is hyperbolic; see \cite{exotic-data} for documentation.  Hence $V_0(P)$ is a nontrivial JSJ piece of $Y$ and is unique in the JSJ decomposition by hypothesis. It follows that $\phi$ can be isotoped so that $\phi(V_0(P))=V_0(Q)$, hence $\phi$ carries $E \subset Y$ to $E \subset Y'$. Since $\gamma$ lies in $E \subset Y$, it follows that $\phi(\gamma)$ lies in $E \subset Y'$.

Let $\mu\subset \partial E$ be the curve corresponding to a meridian of $K \subset S^3$. Since $\mu$ is the unique curve in $\partial E$ that is nullhomologous in $E$, we may assume $\phi|_{E}(\mu)=\mu$. By filling $\partial E$ along $\mu$, we may extend $\phi|_{E}$ to a homeomorphism $\overline{\phi}:S^1\times S^2\to S^1\times S^2$. 
The filled solid torus is naturally identified with a neighborhood of the knot $\overline{K} \subset S^3_0(C)=S^1 \times S^2$ induced by $K \subset S^3 \setminus C$, hence $\overline{\phi}$ actually gives a homeomorphism of the pair $(S^1\times S^2, \overline{K})$. Let $\overline{\gamma}\subset S^1\times S^2$ denote the image of $\gamma$ under the inclusion $E \hookrightarrow S^3_0(C) = S^1 \times S^2$. Up to isotopy, all homeomorphisms of $S^1\times S^2$ send $S^1\times \{pt\}$ to itself \cite{gluck}. Since $\overline{\gamma}$ is isotopic to $S^1\times \{pt\}$, we see that $\overline{\phi}(\overline{\gamma})=S^1\times \{pt\}$.

We claim this implies that the 2-handle of $X'$ attached along $\phi(\gamma)$ can be canceled with the 1-handle of $X'$. Diagrammatically, the isotopy from $\overline{\phi}(\overline{\gamma})$ to $S^1\times \{pt\}$ can be realized by sequence of Reidemeister moves and slides of $\overline{\phi}(\overline{\gamma})$ across the 1-handle of $X'$ that carries $\overline{\phi}(\overline{\gamma})$ to a meridian of the dotted curve $C$.  Moreover, since we may drag the neighborhood $N(\overline{K})$ through this isotopy of $\overline{\phi}(\overline{\gamma})$,  the isotopy of  $\overline{\phi}(\overline{\gamma})$ induces an isotopy of $\phi(\gamma)$ in the handlebody diagram for $X'$. In particular, it follows that the 2-handle attached along $\phi(\gamma)$ can be canceled with the 1-handle corresponding to the dotted circle $C$, thus $X'$ is the trace of some knot $J$ in $S^3$. Moreover, since $X'$ has the same intersection form as $X_n(P(K))$, we know that $X'$ is the $n$-trace of $J$.

The $n$-framed knot $J \subset S^3$ can be viewed as the image of the 0-framed knot $Q \subset V$ after embedding $V$ into $N(\overline{K}) \subset S^3_0(C)=S^1 \times S^2$ and performing surgery on the framed knot $\overline{\phi}(\overline{\gamma})\subset S^3_0(C)=S^1 \times S^2$. Since $\overline{\phi}(\overline{\gamma})$ is disjoint from the image of $V$ in $ S^3_0(C)=S^1 \times S^2$, it follows that $J$ is a satellite knot with (possibly twisted) pattern $Q$. The companion knot $K' \subset S^3$ for the satellite  $J$ can be viewed as the image of $\overline{K} \subset S^3_0(C)=S^1 \times S^2$ after surgery on $\overline{\phi}(\overline{\gamma})$. Recall from above that the map $\overline{\phi}$ induces a homeomorphism of the pair $(S^1 \times S^2, \overline{K})$. By performing surgeries on the framed knots $\overline{\gamma}$ and $\overline{\phi}(\overline{\gamma})$ in $S^1 \times S^2$, we can extend $\overline{\phi}$ to a homeomorphism between the pairs $(S^3,K)$ and $(S^3,K')$. That is, the satellite knot $J$ has companion $K$.  Since the $n$-framed pushoff of $J$ in $S^3$ must agree with the  image of the 0-framed pushoff of $Q \subset V$,  we must identify $S^1\times\{pt\}\subset \partial V$ with the $n$-framed pushoff of $K$ in $\partial N(K)$. Hence $J$ is the satellite $Q_n(K)$. 
\end{proof}

\subsection{Obstructing diffeomorphisms of knot traces and Mazur manifolds} \label{subsec:obstruction}

Here we establish Theorem \ref{thm:mazur}; the proof uses the trace invariance of $\nu$ (Theorem \ref{thm:nu}), which is proved in \S\ref{sec:HF}. The key technical result of this section, Theorem \ref{thm:PQtraces}, relies on the concordance invariants $\nu$ \cite{oz-sz:rational}, $\tau$ \cite{oz-sz:genus}, and $\epsilon$ \cite{hom:epsilon} defined via knot Floer homology. For definitions and discussion of these invariants, we recommend Hom's survey \cite{hom:survey}; here we simply collect the properties used in the proof of Theorem~\ref{thm:PQtraces}. The invariants $\tau$ and $\nu$ are integer-valued, and  $\epsilon$ takes values in the set $\{-1,0,1\}$. They are related as follows:
\begin{itemize}
\item If $\nu(K)=\tau(K)+1$ and $\nu(-K) = \tau(-K)$, then $\epsilon(K)=-1$.
\item If $\nu(K)=\tau(K)$ and $\nu(-K)=\tau(-K)$, then $\epsilon(K)=0$.
\item If $\nu(K) = \tau(K)$ and $\nu(-K)=\tau(-K)+1$, then $\epsilon(K)=1$.
\end{itemize}
These cases are exhaustive, and when $\epsilon(K)=0$, we have $\tau(K)=\nu(K)=\nu(-K)=0$. In \cite{levine:PL}, Levine gives a formula for $\tau(P(K))$ and $\epsilon(P(K))$ in terms of $\tau(K)$ and $\epsilon(K)$:

\begin{thm}[\cite{levine:PL}]\label{thm:levine}
For any knot $K \subset S^3$, 
\begin{equation*} 
\tau(P(K))=\begin{cases}  
\tau(K)+1 & \text{if $\tau(K)>0$ or $\epsilon(K)=-1$} \\
\tau(K) & \text{if $\tau(K) \leq 0$ and $\epsilon(K) \in \{0,1\}$}
\end{cases}
\end{equation*}
and
\begin{equation*}
\epsilon(P(K)) = \begin{cases} 0 & \text{if $\tau(K)=\epsilon(K)=0$} \\ 1 & \text{otherwise.} \end{cases}
\end{equation*}
\end{thm}

This determines an analogous formula for $\nu$, whose proof is left to the reader.

\begin{cor}\label{cor:nu-formula}
For any knot $K \subset S^3$,
\begin{equation*}
\nu(P(K))=\begin{cases} \nu(K)+1 & \text{if $\nu(K)=\tau(K)>0$} \\
\nu(K) & \text{otherwise.} \end{cases}
\end{equation*}
\end{cor}


We will also use the following twist inequality for $\nu$, whose proof is given in \S\ref{sec:HF}.

\begin{prop}\label{prop:nutwist}
Suppose $K$ and $J$ are knots in $S^3$ and $J$ is obtained from $K$ by adding a positive full twist about algebraically 1 strands. Then $\nu(J)\le\nu(K)$. 
\end{prop}

Here ``twisting about algebraically 1 strands'' means that we find an oriented disk $D$ embedded in $S^3$ that intersects $K$ transversely in algebraically $\pm 1$ point, then we replace the trivial braid $K\cap (D\times [-\varepsilon,\varepsilon])$ by a full positive twist.

\begin{thm}\label{thm:PQtraces}
For any knot $K \subset S^3$ with $\nu(K)=\tau(K)>0$ and any integer $n\leq 0$, the knot traces $X_n(P_n(K))$ and $X_n(Q_n(K))$ are homeomorphic but not diffeomorphic.
\end{thm}

\begin{proof}
That $X_n(P_n(K))$ is homeomorphic to $X_n(Q_n(K))$ follows from Proposition~\ref{prop:homeo}. Since $Q_n(K)$ is concordant to $K$, we have $\nu(Q_n(K))=\nu(K)$. On the other hand, we have $\nu(P_n(K)) \geq \nu(P(K))$ by Proposition~\ref{prop:nutwist} and the hypothesis $n \leq 0$, and $\nu(P(K))=\nu(K)+1$ by Corollary~\ref{cor:nu-formula}. Since $\nu(P_n(K))\neq \nu(Q_n(K))$ and $\nu(P_n(K)) > 1$, Theorem~\ref{thm:nu} implies that the traces $X_n(P_n(K))$ and $X_n(Q_n(K))$ are not diffeomorphic.
\end{proof}

These ingredients can now be assembled to   produce exotic pairs of Mazur manifolds. 

\begin{proof}[Proof of Theorem~\ref{thm:mazur}]  
We begin by constructing an infinite family of links $L_m = K_m \cup C$ satisfying the hypotheses preceding Theorem \ref{thm:reduce}, with the additional property that $\nu(K_m) = \tau(K_m)>0$. Let $K_1 \cup C=K \cup C$ be the link depicted in Figure~\ref{fig:tref-phone}. For $m>0$, define $L_m=K_m \cup C$ so that $K_m$ is obtained from $K$ by adding $(m-1)$ positive full twists to the strands passing through the curve $\alpha$. (Note that this is equivalent to viewing $K \cup C$ in the modified surgery diagram for $S^3$ given by $-1/(m-1)$-surgery on $\alpha$.) The knot $K_m$ is the torus knot $T_{2,2m+1}$, so $\nu(K_m)=\tau(K_m)=m>0$ \cite{oz-sz:genus}. 

It remains to check that $L_m=K_m \cup C$ satisfies the hypotheses preceding Theorem~\ref{thm:reduce}. We begin with $K_1 \cup C$: Using SnapPy and Sage \cite{snappy,sagemath}, we may verify that $S^3_0(C) \setminus K_1$ is hyperbolic with $\operatorname{vol}( S^3_0(C) \setminus K_1) \approx 9$. It follows that $S^3_0(C) \setminus K_1$ is not diffeomorphic to the solid torus $V$. And while $V_0(P)$ is hyperbolic, it has $\operatorname{vol}(V_0(P))<4$, so $S^3_0(C) \setminus K_1$ is not diffeomorphic to $V_0(P)$. (See \cite{exotic-data} for documentation regarding these calculations.)

To address $K_m$ with $m>1$, consider the link $K \cup C\cup \alpha$. Using SnapPy and Sage, we may verify that $S^3_0(C) \setminus( K \cup \alpha)$ is hyperbolic with $\operatorname{vol}(S^3_0(C) \setminus( K \cup \alpha)) \approx 11$. To obtain $S^3_0(C) \setminus K_m$, we perform $-1/(m-1)$-surgery on $\alpha$ in $S^3_0(C) \setminus K$. By Thurston's hyperbolic Dehn surgery theorem \cite{thurston:notes} and \cite[Theorem~1A]{neumann-zagier}, there exists $N>0$ such that for $m>N$, the result of $-1/(m-1)$-surgery on $\alpha \subset S^3_0(C) \setminus K$ is hyperbolic and its volume increases monotonically with $m$, converging to $\operatorname{vol}(S^3_0(C) \setminus( K \cup \alpha)) \approx 11$. Since the result of $-1/(m-1)$-surgery on $\alpha \subset S^3_0(C) \setminus K$ is diffeomorphic to $S^3_0(C) \setminus K_m$, we  conclude as above that $S^3_0(C) \setminus K_m$ is not diffeomorphic to $V$ or $V_0(P)$ for $m > N$.

This shows that all but finitely many of the links $L_m=K_m \cup C$ satisfy both the hypotheses preceding Theorem~\ref{thm:reduce} and $\nu(K_m)=\tau(K_m)=m>0$. Applying Theorem~\ref{thm:reduce} and Theorem~\ref{thm:PQtraces}, we conclude that for all $n\le 0$ the corresponding family of Mazur manifolds $W_{L_m,n}$ and $W_{L_m,n}'$ contains infinitely many distinct exotic pairs.
\end{proof}

\section{Property R and integer homology spheres}\label{subsec:propR}

We now deduce Corollary~\ref{cor:propR} from Theorem~\ref{thm:mazur}, aided by a simple lemma.

\begin{lem}\label{lem:duals}
Suppose $(K,\lambda_K)$ and $(K',\lambda_{K'})$ are framed knots in a 3-manifold $Y$ with $M = Y_{\lambda_K}(K)=Y_{\lambda_{K'}}(K')$. Let $k, k' \subset M$ denote the surgery duals to $K,K' \subset Y$, and let $\lambda_k, \lambda_{k'}$ the dual framings of $k,k' $ so that $Y= M_{\lambda_k}(k) =M_{\lambda_{k'}}(k')$. If there is a diffeomorphism of pairs $(Y,K) \cong (Y,K')$ carrying $\lambda_K$ to $\lambda_{K'}$, then there is a diffeomorphism of pairs $(M,k) \cong (M, k')$ carrying $ \lambda_k$ to $ \lambda_{k'}$.
\end{lem}

\begin{proof}
For any knot $J$, we let $\mu_J$ denote the meridian to $J$ and let $\nd(J)$ denote an open tubular neighborhood of $J$. Note that, under the natural identifications $Y \setminus \nd(K) \to M \setminus \nd(k)$ and $Y \setminus \nd(K') \to M \setminus \nd(k')$, we have $\mu_K \mapsto  \lambda_k$, $\mu_{K'} \mapsto \lambda_{k'}$, $\lambda_K \mapsto  \mu_k$, and $\lambda_{K'} \mapsto  \mu_{k'}$. The diffeomorphism $\varphi$ from $(Y,K)$ to $(Y,K')$ induces a diffeomorphism from $Y \setminus \nd(K)$ to $Y \setminus  \nd(K')$ carrying $\lambda_K$ to $\lambda_{K'}$ and $\mu_K$ to $\mu_{K'}$. This induces a diffeomorphism from $M \setminus \nd(k)$ to $M \setminus \nd(k')$ carrying $ \mu_k$ to $\mu_{k'}$ and $\lambda_k$ to $\lambda_{k'}$. By extending this over the Dehn fillings along $\mu_k$ and $\mu_{k'}$, we obtain a diffeomorphism from $M$ to itself carrying $k$ to $k'$ and $\lambda_k$ to $\lambda_{k'}$.
\end{proof}

\begin{proof}[Proof of Corollary~\ref{cor:propR}]
Let $Y$ bound a pair of exotic Mazur manifolds $W$ and $W'$ each obtained from attaching a 2-handle along knot $k$ (resp $k'$) in $S^1 \times S^2$ which generates $H_1(S^1\times S^2)$ with framing $\lambda_k$ (resp $\lambda_{k
'}$). Let $K, K' \subset Y$ denote the surgery duals to $k, k' \subset S^1 \times S^2$, and let $\lambda_K$ and $\lambda_{K'}$ be the dual surgery framings of $K$ and $K'$. Observe that $Y_{\lambda_K}(K)$ and $Y_{\lambda_{K'}}(K')$ are diffeomorphic to $S^1 \times S^2$.

For the sake of contradiction, suppose there exists a diffeomorphism of $Y$ carrying $(K,\lambda_K)$ to $(K',\lambda_{K'})$. By Lemma~\ref{lem:duals}, this induces a diffeomorphism of $S^1 \times S^2$ carrying $(k, \lambda_k)$ to $(k', \lambda_{k'})$. Any diffeomorphism of $S^1 \times S^2$ extends over $S^1 \times B^3$ \cite{laudenbach-poenaru}, so we obtain a diffeomorphism of $S^1 \times B^3$ carrying $(k ,  \lambda_k)$ to $(k', \lambda_{k'})$. This diffeomorphism then extends to the Mazur manifolds $W$ and $W'$ obtained by attaching 2-handles to $S^1 \times B^3$ along $(k , \lambda_k)$ and $(k', \lambda_{k'})$, respectively, yielding a contradiction. \end{proof}

Two related questions naturally arise: Which integer homology spheres admit an $S^1\times S^2$ surgery? And are there integer homology spheres other than $S^3$ that admit a unique $S^1\times S^2$ surgery? As a first step we answer these questions for L-spaces. Recall that a rational homology 3-sphere $Y$ is called an L-space if, for every $\spinc$ structure $\s$ on $Y$, $\widehat{HF}(Y,\s)$ has rank one. 

It has been conjectured that the only integral homology sphere L-spaces are $S^3$ and the Poincar\'e homology sphere; see \cite{hedden-levine}. The following quick verification of this conjecture among homology spheres that bound Mazur manifolds also answers our questions for L-spaces, and has been observed independently by Conway and Tosun \cite{conway-tosun:mazur}.

\begin{prop}
If the integer homology sphere $Y$ is an $L$-space and bounds a Mazur manifold $W$, then the $S^1\times S^2$ surgery on $Y$ is unique and in fact $Y=S^3$ and $W=B^4$.
\end{prop}

\begin{proof}
 By hypothesis, $W$ is built from $S^1 \times B^3$ by attaching a single 2-handle along a knot $K$ in $S^1 \times S^2$ which generates $H_1(S^1\times S^2)$ . Thus the boundary $Y=\partial W$ is an L-space obtained from $S^1 \times S^2$ by Dehn surgery on $K$. By work of Ni and Vafaee \cite[Theorem~1.1]{ni-vafaee}, the complement of $K$ is fibered. As observed in \cite[Remark~3.2]{ni-vafaee}, it follows by work of Baker, Buck, and Lecuona \cite{baker-buck-lecuona} that $K$ is isotopic to a spherical braid. The only spherical braid in $S^1 \times S^2$  which generates $H_1$ is $S^1\times \{pt\}$. We conclude that $K$ is isotopic to $S^1\times \{pt\}$, hence $W$ is diffeomorphic to $B^4$ and $Y$ is diffeomorphic to $S^3$.
\end{proof}

As observed by Conway and Tosun \cite{conway-tosun:mazur}, this gives an alternative to Gabai's proof of Property R: Indeed, if a knot $K$ in an L-space $Y$ admits a surgery to $S^1 \times S^2$, then the above argument shows that $Y=S^3$ and that the surgery dual to $K$ is $S^1 \times \{pt\} \subset S^1 \times S^2$. Applying Lemma~\ref{lem:duals}, we conclude that $K \subset S^3$ is the unknot.

\section{Knot trace invariants and knot Floer homology}
\label{sec:HF}

This section develops the Floer-theoretic input for our main results; our first objectives are to complete the two remaining steps in the proof of Theorem \ref{thm:mazur}: the trace invariance of $\nu$ (\S\ref{subsec:nu}) and the twist inequality for $\nu$ (\S\ref{subsec:twist}). In \S\ref{subsec:other}, we discuss other concordance invariants derived from knot Floer homology and prove Theorem~\ref{thm:tau}.

\subsection{Knot trace invariance of $\boldsymbol{\nu}$} \label{subsec:nu}

Our main technical tool in the proof of Theorem~\ref{thm:nu} is the mapping cone formula for Heegaard Floer homology of 3-manifolds arising by surgery on a knot; in particular we will rely on the fact that a certain map in the homology exact triangle of the mapping cone corresponds to the map induced in Heegaard Floer homology by a surgery cobordism. We give some details; a full exposition can be found in \cite{oz-sz:integer}.

To a knot $K\subset S^3$ we can associate a sequence of (Heegaard Floer) chain complexes $\{A_s,\ s\in \zee\}$ over the field $\eff = \zee/2\zee$, which are described loosely as the ``hook'' complexes $\{\max\{i, j-s\} =0\}$ (these are subquotients of the filtered knot Floer chain complex $CFK^\infty$). For each $s$ we let $B_s = \{i =0\}$, which is just the complex $\cfhat(S^3)$ (with the filtration given by the knot). We have:
\begin{itemize}
\item For each $s$, the homology of $B_s$ is $\hfhat(S^3) = \eff$.
\item For $|s|$ sufficiently large, $H_*(A_s)$ is also $\eff$: for $s\gg 0$ in fact $A_s = B_s$, while for $s\ll 0$ we have $A_s = \{j= s\}$, which is chain homotopic to $B_s$. 
\item For each $s$ there are natural quotients-followed-by-inclusions $v_s: A_s \to \{i = 0\} = B_s$ and $h_s: A_s \to \{j = s\} \simeq B_s$. Here ``$\simeq$'' indicates chain homotopy equivalence; since we are concerned only with maps on homology and the isomorphism $H_*\{j=s\} \cong H_*(B_s) \cong \eff$ is unique, we need no additional information about the equivalence. If $s$ is sufficiently large then $v_s$ is an isomorphism while $h_s$ is 0, and if $s$ is sufficiently negative then $v_s =0$ while $h_s$ is an isomorphism.
\item In fact, we have that $v_s$ is surjective in homology if and only if the same is true of $h_{-s}$, and moreover if $v_s$ is surjective for some $s$ then so is $v_{s'}$ for all $s'>s$. 
\end{itemize}

By definition, $\nu(K)$ is the smallest value of $s$ for which $v_s$ is surjective in homology.

Now fix an integer $n$. We define a linear map $D_n: \bigoplus A_s \to \bigoplus B_s$ as follows. In symbols, if $(s, x)\in \bigoplus A_s$ (where the first entry indicates the summand in which $x$ lies), then $D_n(s,x) = (s, v_s(x)) + (s+n, h_s(x))$. Diagrammatically, $D_n$ is given by the arrows below:

\begin{equation}\label{eq:cone}
\begin{tikzcd}
A_{s-1} \ar[ddd, xshift=-4.5mm] \ar[rrrrddd, xshift=-9.5mm]
&     A_s      \ar[ddd, crossing over, xshift=-14mm]  \ar[rrrrddd, xshift=-18.5mm]               & A_{s+1}  \ar[ddd, crossing over, xshift=-22.5mm]  \ar[rrrrddd, xshift=-28mm] & \cdots     & A_{s+n-1}  \ar[ddd, crossing over, xshift=-32mm]  & A_{s+n}   \ar[ddd, crossing over, xshift=-37mm]  & A_{s+n+1}  \ar[ddd, crossing over, xshift=-41mm]  \\ \\
\\
B_{s-1}          & B_{s}                    & B_{s+1}             & \cdots  & B_{s+n-1} & B_{s+n} & B_{s+n+1} 
\end{tikzcd}
\end{equation}

\noindent Here each $A_s$ is the source of both a vertical and diagonal arrow (and the diagonals go to the left if $n<0$). 

The mapping cone of the chain map $D_n$ is written $\X_n$, and it is known \cite{oz-sz:integer} that the homology of $\X_n$ is isomorphic to the Heegaard Floer homology $\hfhat(S^3_n(K))$. Observe that there are natural chain maps (inclusions) $B_s \to \X_n$ for each $s$, whose maps in homology correspond under this isomorphism to maps $\hfhat(S^3)\to \hfhat(S^3_n(K))$. {\OnS} showed that under the quasi-isomorphism between $\X_n$ and $\cfhat(S^3_n(K))$ the map in homology induced by the inclusion of $B_s$ into the mapping cone corresponds to the map $\hfhat(S^3)\to \hfhat(S^3_n(K))$ induced by the surgery cobordism $X_n: S^3\to S^3_n(K)$, equipped with the \spinc structure $\s_s$ characterized by the property
\begin{equation}\label{spincnumbering}
\langle c_1(\s_s), \sigma\rangle + n = 2s,
\end{equation}
where $\sigma \in H_2(W_n;\zee)$ is a generator.

The following is a trivial adaptation of \cite[Lemma~4.2]{mark-tosun:naturality}. 

\begin{lem}\label{epsilonlemma}
The maps $v_s, h_s: A_s\to B$ induce the same nontrivial map in homology if and only if $ s = \epsilon(K) = 0$.
\end{lem}

Note that under the conditions of the lemma, we also have $\nu(K) = \nu(-K) = \tau(K) = 0$.

\begin{prop}\label{mapchar} Let $K\subset S^3$ be a knot and $n$ a fixed integer, and let $\s$ be a  \spinc structure on the cobordism $X_n(K)$, inducing the map $F_\s:\hfhat(S^3)\to \hfhat(S^3_n(K))$.
\begin{enumerate}
\item If  $|\langle c_1(\s), \sigma\rangle | > 2\nu(K) - n$, then $F_\s$ is zero.
\item If $|\langle c_1(\s), \sigma\rangle | < 2\nu(K) - n$, 
 then $F_\s$ is nonzero.
\item If $|\langle c_1(\s), \sigma\rangle | =2\nu(K) - n$, then $F_\s$ is nonzero if and only if $\epsilon(K) = 0$ and $n\leq 0$.
\end{enumerate}
\end{prop}

\begin{proof}
The map induced by a \spinc structure $\s$ on $X_n(K)$ is given by the inclusion of $B_s$ into the mapping cone, where $s = \frac{1}{2}(\langle c_1(\s),\sigma\rangle + n)$. The inequality $|\langle c_1(\s), \sigma\rangle | < 2\nu(K) - n$ is then equivalent to 
\[
-\nu(K) +n < s < \nu(K),
\]
which means that the homology of $B_s$ is not in the image of any $v_k$ or $h_k$. Hence its homology survives in the homology of the mapping cone, giving that the corresponding cobordism map is nontrivial.

Now suppose $|\langle c_1(\s), \sigma\rangle | \geq 2\nu(K) - n$. By conjugation invariance it suffices to consider the case $\langle c_1(\s), \sigma\rangle \geq 2\nu(K) -n$, which is equivalent by definition of $\nu$ to the assumption that $H_*(B_s)$ is in the image of $v_s$. It follows that $H_*(B_s)$ is in the image of the map on homology induced by $D_n$, as long as there is an element $y_s\in A_s$ with $v_s(y_s) $ generating homology of $B_s$ and $h_s(y_s) = 0$.  This is obvious if $s > -\nu(K)$, for then $h_s$ is the zero map, so let us assume $s\leq -\nu(K)$. If $v_s$ and $h_s$ do not induce the same surjection $H_*(A_s)\to H_*(B) = \eff$ then the desired element can be found. (Since ${h_s}$ and ${v_s}$ do not induce the same surjection then either there exists an element $x\in A_s$ with $({h_s})_*(x)=0$ and  $({v_s})_*(x)\neq 0$ --- in this case we are very pleased indeed --- or there exists an element $x$ in $A_s$ such that $({h_s})_*(x)\neq 0$ and $({v_s})_*(x)=0$. In the latter case, there is also some $x'$ in $A_s$ with both $({v_s})_*(x')\neq 0$ and $({h_s})_*(x')\neq 0$, and then $x+x'$ suffices.)  By Lemma \ref{epsilonlemma}, $v_s$ and $h_s$ do not induce the same surjection unless $s =\nu(K) = 0$, so the first statement of the proposition (in which $s > \nu(K)$) follows. 

Finally consider the last case in the statement, where $s = \nu(K)$ (after conjugating if necessary). By the argument in the previous paragraph, if $\epsilon(K) \neq 0$ then the induced map is trivial. Suppose $\epsilon(K) = 0$ and $n >0$: we claim the induced map $H_*(B_0) \to H_*(\X_n)$ is still trivial. Let $x_0 $ be the generator of $H_*(B_0) = \eff$. Since $\nu(K) =0$, there exists  $y_0\in H_*(A_0)$ with $v_0(y_0) = x_0$. By Lemma \ref{epsilonlemma} the image $h_0(y_0) = x_n\in H_*(B_n)$ is nonzero, and since $n>0$ we know $v_n$ maps $H_*(A_n)$ onto $H_*(B_n)$. Moreover, the map $h_n: H_*(A_n)\to H_*(B)$ is trivial because $\nu(K) =0$,  so for any $y_n\in H_*(A_n)$ such that $v_n(y_n) = x_n$, we have  $D_n(y_0 + y_n) = x_0$. Therefore the inclusion $H_*(B_0)\to H_*(\X_n)$ is trivial. 

To complete the proof we must show that if $\epsilon(K) = 0$ and $n\leq 0$ then $H_*(B_0)\to H_*(\X_n)$ is nonzero. Let $x_0$ be the generator of $H_*(B_0)$ as before, and assume $n< 0$. Then $x_0$ is in the image of $v_0$ but not of $h_{-n}$ (which is the zero map since $-n> -\nu(K) = 0$). For any $y_0\in H_*(A_0)$ with $v_0(y_0) = x_0$, we have $h_0(y_0) = x_{n}\neq 0 \in H_*(B_n)$, by Lemma \ref{epsilonlemma}. Moreover, $x_n$ is not in the image of $v_n$, since the latter is the zero map. It follows that for any nonzero element of the image of $D_n$, if the component in $H_*(B_0)$ is nontrivial then so is the component in $H_*(B_n)$, and in particular $H_*(B_0)\cap im(D_n) = 0$. Hence the induced map is nontrivial when $n < 0$, as desired. Finally if $n = 0$, the complex $\X_n$ splits, and the subcomplex containing $B_0$ consists of the cone on the map $v_0 + h_0: A_0\to B_0$. By Lemma \ref{epsilonlemma} this sum vanishes in $\eff$-coefficient homology, so the inclusion of $H_*(B_0)$ is nontrivial.
\end{proof} 

We will use $F_s$ to denote the map $F_{\s}$ in \spinc structure $\s_s$ with $\langle c_1(\s_s), \sigma\rangle + n = 2s$. Proposition \ref{mapchar} demonstrates that the set of values $s \in \mathbb{Z}$ with $F_s \neq 0$ is a (possibly empty) interval of integers centered at $n/2$. If this set is nonempty, we denote its maximum by $s_{\max}$ (else we say $s_{\max}$ is not defined). 

\begin{cor}\label{cor:smax}
If $F_s=0$ for all $s \in \mathbb{Z}$, then $\nu(K) \leq n/2$ if $n$ is even and $\nu(K)\leq (n+1)/2$ if $n$ is odd.  Otherwise, 
\begin{equation}\label{nubound}
s_{\max} \leq \nu(K) \leq s_{\max}+1,
\end{equation}
with $\nu(K)=s_{\max}$ if and only if $\epsilon(K)=\nu(K)=0$ and $n \leq 0$.
\end{cor}

\begin{proof}
Suppose the maps $F_s$ are trivial for all $s \in \mathbb{Z}$. Proposition~\ref{mapchar}(2) implies that $2\nu(K)-n \leq |\langle c_1(\s) , \sigma\rangle |$ for all \spinc structures $\s$. Because the set of Chern numbers $\langle c_1(\s) , \sigma\rangle$ is the set of integers whose parity agrees with that of $n$,  the minimum value of $|\langle c_1(\s) , \sigma\rangle |$ is 0 if $n$ is even and 1 if $n$ is odd. The stated bound on $\nu$ follows.

If there exist nontrivial maps $F_s$, then $s_{\max}$ is defined. By conjugation invariance of maps induced by \spinc cobordisms, the \spinc structure $\mathfrak{m}$ with $\langle c_1(\mathfrak{m}), \sigma\rangle + n = 2s_{\max}$ has $\langle c_1(\mathfrak{m}),\sigma \rangle\geq 0$, and by Proposition~\ref{mapchar}(1) we have $s_{\max} \leq \nu(K)$ (cf.~\eqref{spincnumbering}).  On the other hand, since $F_{s_{\max}+1}=0$, Proposition~\ref{mapchar}(2) implies that $s_{\max}+1 \geq \nu(K)$.

By the definition of $s_{\max}$, if $\nu(K) = s_{\max}$ then $F_\nu \neq 0$. On the other hand if $F_\nu \neq 0$ then $s_{\max}\ge \nu(K)$, but we have just seen that $\nu(K)\leq s_{\max}$ in general. Hence $\nu(K) = s_{\max}$ if and only if $F_\nu \neq 0$. But Proposition~\ref{mapchar}(3)  implies that $F_\nu \neq 0$ if and only if $\epsilon(K)=\nu(K)=0$ and $n \leq 0$.
\end{proof}

\begin{proof}[Proof of Theorem~\ref{thm:nu}] Our proof proceeds in two steps: When $n \geq 0$, we show that the oriented knot trace $X=X_n(K)$ determines $\nu(K)$; when $n <0$, we show that  either $X$ determines $\nu(K)$ or $\nu(K)\in\{0,1\}$.

To begin, observe that the integers $n$ and $s_{\max}$ are determined by the oriented manifold $X$.  In addition, $X$ determines its orientation reversal $-X \cong X_{-n}(-K)$. We denote \spinc structures on $-X$ by $\s'$ and their corresponding labels determined by \eqref{spincnumbering}  by $s' \in \mathbb{Z}$. As above, we denote the largest $s'$ with $F_{-X,s'}$ nonzero by $s'_{\max}$. We will extract information about $\nu(K)$ and $|\epsilon(K)|$ from the triple of integers $n$, $s_{\max}$, and $s'_{\max}$.

First, suppose $n \geq 0$. We begin with the case in which $X$ admits at least one \spinc structure inducing a nontrivial map: in this case the proof amounts to deciding which of the two possibilities allowed by \eqref{nubound} hold. From the last line of Corollary \ref{cor:smax}, if $n >0$ or $s_{\max}\neq 0$, then $\nu(K)=s_{\max}+1$. In this case it also follows that $\epsilon(K)\neq 0$. 

On the other hand, if $s_{\max}=0$ and $n=0$, then either $\epsilon(K)=\nu(K)=s_{\max}=0$ or $\epsilon(K) \neq 0$ and $\nu(K)=s_{\max}+1=1$. We claim that the possibilities $\epsilon(K)=0$ and $\epsilon(K)\neq0$ can be distinguished by considering $s'_{\max}$. If $\epsilon(K)=0$, then $\epsilon(-K)=0$ and $\nu(-K)=0$. Since $n = \nu(-K) = 0$, it follows from Proposition \ref{mapchar}(3) that $s_{\max}'= 0$, in particular $s_{\max}'$ is defined. Otherwise,  if $\epsilon(K)\neq 0$ then $\epsilon(-K)\neq 0$ and $\nu(-K) = 0$, in which case Proposition \ref{mapchar} implies that there are no \spinc structures on $-X$ inducing nontrivial maps, i.e. $s_{\max}'$ is undefined. Summarizing, we find that if $s_{\max} = 0$ and $n = 0$, then $\nu(K) = 0$ if $s'_{\max}$ is defined (and in this case $\epsilon(K) = 0$), and otherwise $\nu(K) = 1$ (and $\epsilon(K) \neq 0$). Thus, $\nu(K)$ and $|\epsilon(K)|$ are determined by the invariants $s_{\max}$, $s'_{\max}$ and $n$ associated to $X_n(K)$, as long as $s_{\max}$ is defined.

We next consider the case in which $-X\cong X_{-n}(-K)$ admits at least one \spinc structure inducing a nontrivial map, i.e.~$s'_{\max}$ is defined.  
Suppose $s'_{\max}\neq0$. If $\epsilon(-K)=0$, then $\nu(-K)=0$. But $-n\le 0$ by hypothesis, thus Corollary~\ref{cor:smax} implies that $\nu(-K)=s'_{\max}\neq 0$, a contradiction. Thus when $s'_{\max}\neq0$, we must have $\epsilon(-K)\neq 0$. This implies $\nu(-K) = s'_{\max}+1$ and $\nu(K)=-\nu(-K)+1=-s'_{\max}$. Supposing instead now that $s'_{\max}=0$, we claim $\nu(K)=0$. Indeed, if $\nu(-K)=s'_{\max}=0$, then $\epsilon(-K)=0$ and therefore $\nu(K)=0$. And if $\nu(-K)=s'_{\max}+1=1$, then $\epsilon(-K)\neq 0$ and therefore $\nu(K)=-\nu(-K)+1=0$. (Note that when $s'_{\max} =0$ and $s_{\max}$ is not defined,  we cannot determine $|\epsilon(K)|$.)

Finally, consider the case in which neither $X$ nor $-X$  admits \spinc structures inducing nontrivial maps. First, we claim that $\epsilon(K) \neq 0$. Indeed, if $\epsilon(K)=0$ then $\nu(-K)=\epsilon(-K)=0$ and therefore Proposition~\ref{mapchar}(3)  implies that that the map induced by the \spinc structure $\s'$ on $-X_n(K)$ satisfying $\langle c_1(\s'), \sigma\rangle = n$ is nontrivial, a contradiction.  Next, we claim that this case only arises if $n$ is odd. Indeed, if $n$ is even, then applying the first part of Corollary~\ref{cor:smax} to $X$ and $-X$ yields the bounds $\nu(K) \leq n/2$ and $-\nu(K)+1=\nu(-K) \leq -n/2$, which combine to give $n/2+1 \leq \nu(K) \leq n/2$, a contradiction. Thus we may assume $n$ is odd. In this case, an analogous argument yields the inequality
$$\tfrac{1}{2}(n+1)=-\tfrac{1}{2}(-n+1)+1 \leq \nu(K) \leq \tfrac{1}{2}(n+1).$$
This implies $\nu(K)=(n+1)/2$, completing the argument in the case where $n \geq 0$.

For $n<0$, we may shift to considering the mirror $-K$. Since $-n>0$, the preceding argument shows that $-X=X_{-n}(-K)$ determines $\nu(-K)$. Moreover, there is only one case in which we fail to determine $|\epsilon(-K)|$ and in this case $\nu(-K)=0$. When we can determine both $\nu(-K)$ and $|\epsilon(-K)|$, we can determine $\nu(K)$ as well. In the exceptional case, since $\nu(-K)=0$, we must have that $0\leq \nu(K) \leq 1$. It follows then that for a pair of knots $K$ and $K'$ with $X_n(K)\cong X_n(K')$, we can only have  $\nu(K)\neq\nu(K')$ if $\{\nu(K),\nu(K')\}=\{0,1\}$ and $n<0$. 
\end{proof}

The information deduced in the preceding argument is summarized in the following table, in which information about $X_n(K)$ is listed on the left with corresponding deductions about $\nu(K)$ and $|\epsilon(K)|$ on the right. A numerical condition on $s_{\max}$ or $s_{\max}'$ implicitly assumes the quantity is defined; an empty entry means no assumption is made.

\bigskip
\begin{center}
\begin{tabular}{c|c|c||c|c}
$s_{\max}$ & $s_{\max}'$ & $n$ & $\nu(K)$ & $|\epsilon(K)|$ \\ \hline\hline

$\neq 0$ & &  &$s_{\max} + 1$ & 1 \\ 
defined & & $>0$ &  $s_{\max} + 1$ & 1 \\ 
0 & defined & 0 & 0 & 0 \\ 
0 & undef & 0 & 1 & 1\\ 
undef & 0 & $>0$ & 0 & undet. \\ 
undef & undef &  & $\frac{1}{2}(n+1)$ & 1 \\
& $\neq 0$  &  & $-s_{\max}'$ & 1 \\ 
 & defined & $< 0$ & $-s_{\max}'$ & 1 \\ 
0 & undef & $<0$ & undet & undet. 
\end{tabular}
\end{center}
\bigskip

We remark that the triple $\{n, s_{\max}, s'_{\max}\}$ is indeed insufficient to determine $\nu(K)$ in the case where we draw no conclusion. For example, the right-handed trefoil knot $K$ has  $\nu(K) = \epsilon(K) = 1$ and the figure-eight knot $K'$ has $\nu(K') =\epsilon(K')= 0$, but Lemma \ref{mapchar} shows that for any $n<0$ both knots have $s_{\max}  =0$ and $s_{\max}'$ undefined. Indeed, the set of (Chern numbers of) \spinc structures inducing nontrivial maps on $X$ and $-X$ for these two knots is the same. However we are unaware of any such examples having diffeomorphic $n$-traces. Therefore, it seems reasonable to conjecture that $\nu$ is a trace invariant for all framings. In fact, we can use the twist inequality for $\nu$ (stated as Proposition~\ref{prop:nutwist} and proved in the next subsection) to rule out potential counterexamples, as follows.

The primary construction used to produce pairs of knots with diffeomorphic traces is often called \emph{dualizable patterns},  and was pioneered by Akbulut in \cite{akbulut2dclasses} and developed by Lickorish in \cite{lickorishshake}. To the authors' knowledge, all pairs $K,K'$ in the literature with $X_n(K)\cong X_n(K')$ can be produced this way. We will not recall the construction here; for a recent treatment see \cite{miller-picc,picc:shake}.  When this construction is used to produce a pair of knots $K$ and $K'$ with diffeomorphic $n$-traces, then there exists a (dualizable) pattern $P$ with $P(U)\isoeq K$ and a dual pattern $P^*$ with $P^*_n(U)\isoeq K'$ and $X_m(P(U)) \cong X_m(P^*_{m}(U))$ for all $m \in \mathbb{Z}$.  

\begin{prop}\label{prop:nudual}
For any $n \in \zee$, if $K$ and $K'$ have diffeomorphic $n$-traces constructed using dualizable patterns, then $\nu(K)=\nu(K')$. 
\end{prop}

\begin{proof}
Let $P$ be the dualizable pattern with $P(U)\isoeq K$ and $P^*_n(U)\isoeq K'$. Then $X_0(P(U))\cong X_0(P^*(U))$, so $\nu(P(U))=\nu(P^*(U))$ by Theorem~\ref{thm:nu}. Further, $X_0(P^*_n(U))\cong X_0(P_{-n}(U))$ (cf.~\cite[Proposition~3.9]{miller-picc}), so $\nu(P^*_n(U))=\nu(P_{-n}(U))$. 

If $n \geq 0$ then using these identities together with the twist inequality we obtain
\begin{equation*}
\nu(K) = \nu(P(U))  \leq  \nu(P_{-n}(U))
= \nu(P^*_n(U)) = \nu(K') 
\leq  \nu(P^*(U))
= \nu(P(U)) = \nu(K),
\end{equation*}
and similarly if $n\leq 0$ then
\begin{equation*}
\nu(K) = \nu(P(U)) = \nu(P^*(U)) 
\leq  \nu(P^*_{n}(U)) = \nu(K') 
= \nu(P_{-n}(U)) 
\leq  \nu(P(U)) = \nu(K). \qedhere
\end{equation*}
\end{proof}

\subsection{A twist inequality for $\boldsymbol{\nu}$} \label{subsec:twist}

Let $\Sigma$ be a smooth surface with boundary embedded in a 4-manifold $W$ having $\partial W = S^3$, and such that $b_1(W) = b^+_2(W) = 0$. Then Donaldson's theorem \cite{donald:embedding} shows that there is a basis $e_1,\ldots, e_n$ for $H_2(W)$ such that with respect to the intersection pairing we have $e_i .e_j = -\delta_{ij}$. Writing $[\Sigma]=s_1\cdot e_1+s_2\cdot e_2+\cdots s_n\cdot e_n$, let $|[\Sigma]|:=\sum_{i=1}^n |s_i|$ be the $L^1$ norm of $\Sigma$, which is independent of the choice of basis. We have the following adjunction-type inequality for this situation.

\begin{thm}[Ozsv\'ath-Szab\'o {\cite[Theorem~1.1]{oz-sz:genus}}]\label{taunegdef}
Let W be a smooth, oriented 4-manifold with $b_1(W)=b^+_2(W) = 0 $ and $\partial W= S^3$. Let $\Sigma$ be a smoothly embedded surface of genus $g$ in $W$ whose connected boundary lies on $S^3$, where it is embedded as the knot $K$. Then $2\tau(K) + |[\Sigma]| + [\Sigma]. [\Sigma] \le 2g$.
\end{thm}

A twist inequality for $\tau$ follows quickly.

\begin{cor}\label{tautwist}
    Suppose $K$ and $J$ are knots in $S^3$ and $J$ is obtained from $K$ by adding a positive full twist about algebraically 1 strands. Then $\tau(J)\leq\tau(K)$. 
\end{cor}
\begin{proof}
Observe that $J\#-K$ would be slice in $B^4$ if we removed the positive full twist that distinguishes $-K$ from $-J$. We can remove the twist by blowing up $B^4$, and we see that $J\#-K$ bounds a smooth, properly embedded disk $D$ in $B^4\# \cpbar$ with $|[D]|=1$ and $[D]\cdot[D]=-1$.  Then we can apply Theorem \ref{taunegdef} and the additivity of $\tau$ to conclude $0\ge \tau(J\#-K)=\tau(J)-\tau(K)$. 
\end{proof}

From these results we deduce a (less general) analog of Theorem \ref{taunegdef} for $\nu$:

\begin{thm}\label{thm:nunegdef}
Let W be a smooth, oriented four-manifold with $b_1(W)=b^+_2(W) = 0$ and $\partial W = S^3$. Let $D$ be a smoothly embedded disk in $W$ whose boundary lies on $S^3$, where it is embedded as the knot $K$. Then  $2\nu(K) + |[D]| + [D] \cdot [D ] \le 0$
\end{thm}

\begin{proof}
If $\epsilon(K)\geq 0$ then $\tau(K)=\nu(K)$ and the result follows immediately from Theorem \ref{taunegdef}. Thus we will assume $\epsilon(K)=-1$, which we note implies $\nu(K)=\tau(K)+1$. 

Suppose $K$ bounds a disk $D \subset W$ as in the hypothesis, and let  $r=|[D]|$ and $n=[D]\cdot[D]$. By removing (the interior of) a small tubular neighborhood of a point in the interior of $D$, we get a concordance $A$ in $W\smallsetminus B^4$ from $K$ to $U$. This concordance $n$-frames $K$ and $0$-frames $U$.  Let $P$ denote the Mazur pattern, and define $P'$ to be the annulus $P\times I\subset S^1\times D^2\times I$. Replacing the tubular neighborhood of $A$ with this $S^1\times D^2\times I$ and considering the inclusion of $P'$, we get a concordance $A'$ in $W\smallsetminus B^4$ from $P_n(K)$ to $P(U)$. Since $P(U)$ is slice, we can cap $A'$ off to a disk $D'$ in $W$ with $\partial(D')=P_n(K)$, $|[D']|=r$, and $[D']\cdot[D']=n$.

Applying Theorem \ref{taunegdef} to $D'$, we have $2\tau(P_n(K))+r+n\le 0$. The hypothesis $b^+_2(W)=0$ ensures $n \leq 0$, so Corollary \ref{tautwist} implies  $\tau(P(K))\le \tau(P_n(K))$, hence $2\tau(P(K))+r+n\le 0$. Since $\epsilon(K)=-1$, Theorem \ref{thm:levine} implies that $\tau(P(K))=\tau(K)+1$. Thus $2\tau(K)+2+|[D]|+[D]\cdot[D]\le 0$ and the result follows. 
\end{proof}

The following lemma is easily proved using the additivity of $\tau$ under connected sum, by considering the various cases (also see \cite[Proposition 3.6(6)]{hom:epsilon}). 

\begin{lem}
For all knots $K, J$ in $S^3$, $\nu(K\#J)\le \nu(K)+\nu(J)$. 
\end{lem}

\begin{proof}[Proof of Proposition \ref{prop:nutwist}]
Applying Theorem \ref{thm:nunegdef} exactly the same way as in Corollary \ref{tautwist}, we have that $2\nu(J\#-K)+1-1\le 0$, hence $\nu(J\#-K)\le 0$. Then subadditivity and concordance invariance of $\nu$ allow us to conclude that $\nu(J)=\nu(J\#-K\#K)\le \nu(J\#-K)+\nu(K)\le \nu(K)$.\end{proof}

\subsection{Knot trace invariance of other concordance invariants}
\label{subsec:other}

Many classical concordance invariants are determined by the 3-manifold obtained by zero-surgery on a knot, hence clearly give rise to zero-trace invariants. However, such invariants are not suited to detecting exotic zero-traces because homeomorphic traces have diffeomorphic boundary. The invariant $\nu$ is particularly well-suited to our application precisely because it is not a zero-\emph{surgery} invariant, but is a zero-\emph{trace} invariant.

It is straightforward to show that the concordance invariant $\nu^+$ defined by Hom and Wu in \cite{hom-wu} also has the property of being a zero-trace, but not zero-surgery, invariant; see Remark~\ref{rem:nu+} below. However, one should not be too distressed to find that one's favorite concordance invariant is \emph{not} a zero-trace invariant; concordance invariants which are not zero-trace invariants have the potential to yield refined slicing obstructions; see \cite{picc:conway}. Several concordance invariants are already known \emph{not} to be zero-trace invariants, including Rasmussen's $s$-invariant \cite{picc:shake} and the $d$-invariants of a knot's branched double-cover \cite{miller-picc}.  We now provide examples demonstrating that $\tau$ and $\epsilon$ also fail to be zero-trace invariants.

\bigskip

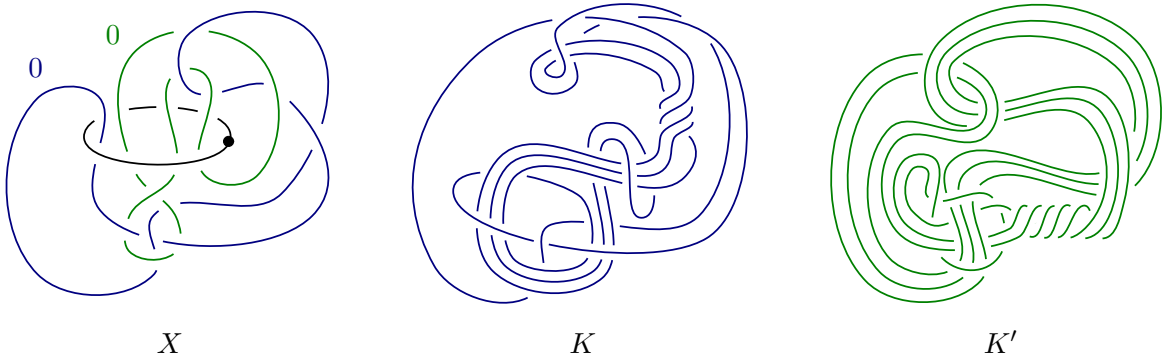
\begin{figure}[h]\center
\def\svgwidth{\linewidth}\input{tauknots.pdf_tex} 
\caption{The zero-traces of the knots $K$ and $K'$ are both diffeomorphic to the 4-manifold given by the Kirby diagram on the left, as can be seen by canceling the 1-handle in $X$ with one or the other 2-handle.}
\label{fig:tau-knots}
\end{figure}

\begin{proof}[Proof of Theorem~\ref{thm:tau}]
Let $X$ be the 4-manifold described by the Kirby diagram in Figure~\ref{fig:tau-knots}. This diagram fits into the setup described in \cite[Theorem~2.1]{picc:shake}, hence describes a pair of diffeomorphic knot traces $X_0(K)$ and $X_0(K')$. 
Using the knot Floer homology calculator \cite{hfk-calc}, we find $\underline{\tau}(K)=1$ and $\underline{\epsilon}(K)=1$ versus  $\underline{\tau}(K')=0$ and $\underline{\epsilon}(K')=-1$; see \cite{exotic-data} for documentation. The claim now follows by appealing to the equivalence between $\underline{\tau}$ and $\tau$ and between $\underline{\epsilon}$ and $\epsilon$ \cite{oz-sz:matchingsHFK}.
\end{proof}

\begin{rem}
\textbf{(a)}  Using \cite{sknotjob}, we calculate the Rasmussen invariants of $K$ and $K'$ (with $\mathbb{F}_2$-coefficients) to be $s(K)=0$ and $s(K')=2$; see \cite{exotic-data}.  
\textbf{(b)}  By tying a local knot into one of the 2-handles in the handlebody diagram of $X$ shown above, one can easily produce additional examples. In particular, tying a right-handed trefoil into the blue 2-handle yields a pair of knots $K$ and $K'$ with $X_0(K) \cong X_0(K')$, $\underline{\tau}(K)=2$, and $\underline{\tau}(K')=1$. 
\end{rem}

\begin{rem}\label{rem:nu+} 
We have focused exclusively on the ``hat'' flavor of Heegaard Floer theory for our results, but one may wonder if additional information can be obtained by considering other flavors. For example, Hom and Wu \cite{hom-wu} studied an invariant $\nu^+$ derived from the ``plus'' version of knot Floer homology (or equivalently formulated via the ``minus'' theory), and showed that for all knots $K$, $\tau(K)\leq \nu(K)\leq \nu^+(K)\leq g_4(K)$, where $g_4(K)$ is the slice genus of $K$ in $B^4$. Moreover, $\nu^+$ is always nonnegative, and the inequality between $\nu$ and $\nu^+$ can be arbitrarily large even when $\tau(K)$ and $\nu(K)$ are positive. 

We do not undertake a detailed study of $\nu^+$ as a trace invariant or shake genus bound here. However, we can make one easy observation in the case of zero-traces, as follows. It was observed by Rasmussen \cite{rasmussenthesis} that $\nu^+$ can be characterized in terms of the zero-surgery cobordism; this is  conveniently expressed in terms of the maps $F^-_{\s}: HF^-(S^3)\to HF^-(S^3_0(K); \s)$ induced by the surgery cobordism $X_0(K)$. Here we use the same symbol for a \spinc structure on $X_0(K)$ and its restriction to $S^3_0(K)$, since the former is determined by the latter. Recall that for any \spinc 3-manifold $(Y,\s)$, the group $HF^-(Y,\s)$ admits an endomorphism $U$, and the reduced subgroup $HF^-_{red}(Y,\s)$ is defined to be the kernel of $U^N$ for any sufficiently large $N$. Then it follows from the discussion in \cite[Section 7]{rasmussenthesis} that
\[
\nu^+(K) = \min\left\{\ts\frac{1}{2}|\langle c_1(\s), \sigma\rangle| : im(F^-_{\s})\cap HF^-_{red}(S^3_0(K), \s) = 0\right\}
\]
where $\sigma$ is a generator of $H_2(X_0(K))$. (In the terminology of \cite{rasmussenthesis}, $\nu^+(K) = \min\{k : h_k = 0\}$ where $h_k$ is the ``local $h$-invariant'' denoted $V_k$ in \cite{hom-wu}; compare \cite[Definition 7.1]{rasmussenthesis}.)

It follows immediately that $\nu^+(K)$ is both a lower bound on zero-shake genus (by the adjunction inequality for maps in $HF^-$), and an invariant of the zero-trace (since the characterization above clearly depends only on the smooth structure of $X_0(K)$). Since $\nu(P(K))$ and $\nu(Q(K))$ give sharp bounds on $g_4$ for $K$ the right-handed trefoil, Proposition \ref{prop:homeo} and Hom and Wu's inequalities show that $\nu^+$ is not a zero-surgery invariant. In addition, Hom and Wu's inequalities  imply that $\tau$ and $\nu$ are lower bounds on the zero-shake genus, however, from this perspective the zero-trace invariance of $\nu$ is not clear. 
\end{rem}

\section{Shake genus bounds from knot Floer homology}\label{sec:bounds}

As mentioned in the introduction, an understanding of the genus function on surgery traces $X_n(K)$ has allowed much of the progress to date in the study of smooth structures on knot traces. Here we prove Theorem \ref{thm:nushakeboundclean}, that $\nu$ gives a lower bound on $n$-shake genus $\gsh^n(K)$ for a certain range of $n$, where we recall that 
$\gsh^n(K)$ is the minimum $g$ such that there exists a smoothly embedded surface of genus $g$ in $X_n(K)$ representing a generator of homology. Previously Celoria-Golla-Levine gave lower bounds on the zero-shake genus of $K$ coming from certain Heegaard Floer correction terms \cite{celoria-golla}. Our theorem has the advantage that it applies to a broader range of framings and that $\nu$ is generally quite computable \cite{oz-sz:matchingsHFK}. 

The essential tool here is the adjunction inequality in Heegaard Floer theory for surfaces embedded in cobordisms. When the surface has nonnegative self-intersection this result is well-established in the literature, but the case of negative square is slightly more delicate. For the sake of a self-contained exposition, we give the statement we need here; it corrects Lemma 3.5 of \cite{oz-sz:genus} and slightly extends a similar correction in \cite{raoux}.

\begin{prop}\label{adjineq} Let $W$ be a compact oriented cobordism and $S$ an embedded surface of genus $g$. If a \spinc structure $\s$ induces a nonzero map on $\hfhat$ then $|\langle c_1(\s), S\rangle | + S\cdot S \leq 2g$. If in addition we have $S\cdot S > -2g(S)$, then $|\langle c_1(s), S\rangle| + S\cdot S \leq 2g - 2$. 
\end{prop}

\begin{proof}
By the composition law for maps on Floer homology, together with a K\"unneth principle (cf. Hedden-Raoux \cite{heddenraoux}), it suffices to consider the case that $W$ is a neighborhood of $S$ with a small ball removed. This neighborhood is a disk bundle over $S$ of degree $n = S\cdot S$, and has a handle description with $2g$ 1-handles and a single 2-handle attached along a nullhomologous knot $K\subset \#^{2g} S^1\times S^2$ with framing $n$. We can describe $K$ as the connected sum of $g$ copies of the knot $B\subset \#^2 S^1\times S^2$ given as one component of the Borromean rings after zero-framed surgery on the other two components (cf. \cite[Section 9]{oz-sz:hfk} and \cite[Section 5.2]{oz-sz:integer}).

The Floer homology of $Y = \#^{2g} S^1\times S^2$ is isomorphic to an exterior algebra on an $\eff$-vector space of dimension $2g$ that can be identified with $H^1(S;\eff)$, and the knot Floer homology of $K$ is also identified with that exterior algebra. However, the Alexander grading on $\widehat{HFK}(Y,K)$ is nontrivial:
\[
\widehat{HFK}(Y,K, j) \cong \Lambda^{g+j} H^1(S).
\]
Moreover, $\widehat{HFK}(Y,K,j)$ is supported entirely in homological degree $j$. The (homology of the) complexes $A_s$ and $B_s$ can be described in these terms as well, since there are no higher differentials in the knot Floer complex for $K$ \cite{oz-sz:hfk}. We have:
\begin{equation}\label{eq:borhfk}
B_s \cong \bigoplus_{j=-g}^g \Lambda^{g+j}_{(j)} \quad \mbox{and} \quad A_s \cong \bigoplus_{j=-g}^s \Lambda^{g+j}_{(j)} \oplus \bigoplus_{i = -1}^{-g+s} \Lambda^{g+s-i}_{(s+i)},
\end{equation}
where we have written $\Lambda^k_{(j)}$ for the space $\Lambda^k H^1(S)$ lying in homological degree $j$, and the symbols $i$ and $j$ in the decomposition above correspond to the $(i,j)$ gradings on $CFK^\infty$. Observe that $A_s \cong B_s$ when $|s|\geq g$.

Now, the 1-handles in $W$ map the generator of $\widehat{HF}(S^3)$ to the generator in top degree of $\hfhat(Y)$, or equivalently to the generator of $\Lambda^{2g}_{(g)}$ of $B_s$. Just as for knots in $S^3$, the map induced by the 2-handle in $W$ can be identified with the map induced by the inclusion of $B_s$ in the cone \eqref{eq:cone}, and therefore we are interested in when the top-degree generator survives in the homology of the mapping cone. 

Now, it can be seen that $v_s$ corresponds to the projection of $A_s$ onto the first summand in the decomposition above followed by inclusion in $B_s$ as the obvious factor, while $h_s$ is given by projection onto the sum of the term $j=s$ plus the second summand, followed by the map $\Lambda^{g+s-i} \stackrel{\sim}{\to} \Lambda^{g-s+i} \hookrightarrow B_s$. Let us write $b^{top}_s$ for the generator in highest degree of $B_s$, and $b^{bot}_s$ for the lowest-degree generator. It is then easy to see that:
\begin{itemize}
\item $b^{top}_s$ is in the image of $v_s$ if and only if $s\geq g$, and in this case $b^{top}_s = v_s(a^{top}_s)$ for a topmost generator of $A_s \cong B_s$.
\item $b^{top}_{s+n}$ is in the image of $h_{s}$ if and only if $s\leq -g$, and in this case $b^{top}_{s+n} = h_s(\tilde{a}^{top}_s)$ for $\tilde{a}^{top}_s$ a generator in top degree of $A_s$. If $s = -g$ then $\tilde{a}^{top}_s$ lies in the summand $j = -g$ of \eqref{eq:borhfk}, and if $s<-g$ then $\tilde{a}^{top}_s$ is in the summand $i = g+s$.
\item When $s\geq g$, the top generator of $A_s$ is annihilated by $h_s$ except in the case $s=g$, and in this case $h_g(a^{top}_g)$ is the bottom generator $b^{bot}_{s+n}$ of $B_{s+n}$. 
\item The lowest generator $b^{bot}_s$ of $B_s$ is in the image of $v_s$ if and only if $s\geq -g$, and in this case is the image of a generator $a_s^{bot}$ in the summand corresponding to $j = -g$ of \eqref{eq:borhfk}. Moreover $a_s^{bot}$ is in the kernel of $h_s$ unless $s = -g$.
\end{itemize}

Studying the mapping cone in a manner just as in the proof of Proposition \ref{mapchar}, we see that if $-(g-1) +n \leq s \leq g-1$ then the topmost generator of $B_s$ is not in the image of any maps in the cone and therefore its image is nonzero in $\hfhat(Y_n(K))$ under the corresponding cobordism map. If $s\geq g+1$ (resp.~$s \leq -(g+1) +n$) then the top generator is  the image under $v_s$ (resp.~$h_{s-n}$) of a generator that is killed by $h_s$ (resp.~$v_{s-n}$) and hence vanishes under the cobordism map. This proves the first claim of the proposition:  the inclusion of the top generator of $B_s$ corresponds to the map associated to a \spinc structure with $\langle c_1(\s), S\rangle + n = 2s$, where $n = S\cdot S$, hence the preceding shows the map is trivial if $|\langle c_1(\s),S\rangle| + S\cdot S \geq 2g+2$ (and nontrivial if the left side is $\leq 2g-2$).

For $s = g$, we claim that the top generator $b_g^{top}\in B_g$ is in the image of the mapping cone differential when $n > -2g$. (A similar argument shows that $b_{-g+n}^{top}$ is also in the image of $D_n$.)  To see this, observe $b_g^{top} = v_g(a_g^{top})$, and $h_g(a_g^{top}) = b_{g+n}^{bot} = v_{g+n}(a_{g+n}^{bot})$, where the last holds since $g+n \geq -g$. Furthermore, $h_{g+n}(a_{g+n}^{bot}) = 0$ when $g+n > -g$. Hence in this case $b_g^{top} = D_n(a_g^{top} + a_{g+n}^{bot})$. 

It follows that when $n > -2g$, the map in $\hfhat$ induced by $W$ vanishes in the \spinc structures corresponding to $s = g$ and $s = -g+n$. Comparing with \eqref{spincnumbering} this means we have vanishing when $|\langle c_1(\s), S\rangle | + S.S = 2g$, so we get the desired stronger conclusion in this case.

One can also check that if $n\leq -2g$ the generators $b_g^{top}$ and $b_{-g+n}^{top}$ survive in the mapping cone, so the conclusion cannot be strengthened in general. \end{proof}

\begin{prop}\label{nushakebound}
Let $K$ be a knot in $S^3$.  
\begin{enumerate}
\item If $n<2\nu(K)-1$ then $\nu(K)-1\le \gsh^n(K)$. 
\item If $-2\gsh^n(K)<n<2\nu(K)-1$ then $\nu(K)\le \gsh^n(K)$. 
\item If $n\ge 0$ and $\gsh^n(K)=0$ then $\nu(K)\le \gsh^n(K)=0$
\end{enumerate}
\end{prop}

\begin{proof} First assume that  $n < 2\nu(K) -1$. Proposition \ref{mapchar} shows that \spinc structure $\s_{\nu(K)-1}$ induces a nontrivial map $F_\s$. Then Proposition \ref{adjineq} implies that $2\nu(K)-2-n+n\le 2 \gsh^n(K)$ and the first claim follows. If we also assume that $-2\gsh^n(K)<n$ then Proposition \ref{adjineq} gives the stronger implication that $2\nu(K)-n+n\le 2 \gsh^n(K)$, and we get the second claim. 

Now we will evaluate the case that $n\ge 0$ and $\gsh^n(K)=0$. Let $S$ denote the sphere in $X_n(K)$ generating second homology. Let $W_n(K)$ denote the cobordism $X_n(K)\setminus N(S)$, where $N(S)$ denotes an open tubular neighborhood of $S$. It is standard to check that $W_n(K)$ is an integer homology cobordism from $S^3_n(U)$ to $S^3_n(K)$ in which the meridian of $U$ is homologous to the meridian of $K$. 

Now consider $U$ in the boundary of $B^4\#_{n} \cpbar $; $U$ bounds a disk $D$ such that $Z:=(B^4\#_{n} \cpbar )\smallsetminus N(D)$ has boundary $S^3_{n}(U)$ and such that in the standard (2-handle) basis for $H_2(B^4\#_{n} \cpbar)$, $D$ represents the $(1,1,\ldots, 1)$ class. Form $Z':=Z\cup W$ by identifying their $S^3_{n}(U)$ boundary components via any homeomorphism taking the meridian of $U$ to the meridian of $U$. By construction, $Z'$ has the homology type of a slice disk exterior for $K$ in $B^4\#_{n} \cpbar$. Then, attaching a 0-framed 2-handle to $Z'$ along the meridian of $K$ in $S^3_{n}(K)$ yields a 4-manifold $Z''$ with $S^3$ boundary and the homology type of $B^4\#_{n} \cpbar$ in which $K$ bounds a disk $D'$ such that in the inherited basis for $H_2(Z'')$, $D'$ represents the $(1,1,\ldots, 1)$ class and such that $D\cdot D=-n$. 
Finally, we appeal to Theorem \ref{thm:nunegdef}, which implies $\nu(K)\le 0=\gsh^n(K)$. 
\end{proof}

\begin{proof}[Proof of Theorem~\ref{thm:nushakeboundclean}]
First consider the case in which $n=0$. If $\gsh^0(K)=0$ this follows from Proposition \ref{nushakebound}(3). If $\gsh^0(K)>0$ then either $\nu(K)\le 0$ in which case the result is vacuously true, or $\nu(K)\ge 1$ in which case Proposition \ref{nushakebound}(2) applies.

Now observe that if $n\le 2\nu(K)-2$ and $n\le -2 \gsh^n(K)$ then $2\nu(K)-2\le 2 \gsh^n(K)\le -n$, where the first inequality follows from  Proposition \ref{nushakebound}(1). Thus $n\le 2-2\nu (K)$. Therefore, by considering the contrapositive, we have that when $2-2\nu(K)<n\leq 2\nu(K) -2$ we must have $n>-2\gsh^n(K)$. Then Proposition~\ref{nushakebound}(2) applies. 
\end{proof}

A quick application of Theorem \ref{thm:nushakeboundclean} is a bound on the ``classical invariants'' of a Legendrian knot in the standard contact structure on $S^3$. As this is incidental to the rest of our results we do not review all the terminology here, but  bounds of this sort have played a central role in recent results on knot concordance and knot traces \cite{cochran-ray,yasui:conc} so we include it here.

Recall that for any Legendrian knot $\K$ in knot type $K$, we have the inequality 
\begin{equation}\label{taubound}
\tb(\K) + |\rot(\K)| \leq 2\tau(K) -1
\end{equation}
relating the Thurston-Bennequin number $\tb(\K)$, rotation number $\rot(\K)$, and the invariant $\tau(K)$ from knot Floer homology (see \cite{plamenevskaya:tau}). The following corollary of Theorem \ref{thm:nushakeboundclean} should be compared to the slice Bennequin inequality of \cite{akbulut-matveyev, lisca-matic} and the shake-slice Bennequin inequality of Cochran-Ray \cite{cochran-ray}.

\begin{cor}\label{cor:shsliceben} Let $\K$ be any Legendrian representative of the knot type $K$ in $S^3$. For any $n\leq 2\nu(K) -2$,
\[
\tb(\K) + |\rot(\K)| \leq 2\gsh^n(K) + 1.
\]

If $n=0$ or $2-2\nu(K)<n\leq 2\nu(K)-2$, then
\begin{equation*}
\tb(\K) + |\rot(\K)| \leq 2\gsh^n(K) -1.
\end{equation*}
\end{cor}

Indeed, since $\tau(K)\leq \nu(K)$, the  statement follows from \eqref{taubound} and Proposition \ref{nushakebound} (or Theorem~\ref{thm:nushakeboundclean} in the second statement). Note that the shake slice Bennequin inequality of Cochran-Ray assumes $n \leq \tb(\K) -1$; from \eqref{taubound} this implies $n \leq 2\nu(K) -2$, so the above extends the prior shake-slice Bennequin inequality.

\section{Further examples}\label{sec:examples}
\addtocounter{subsection}{1}

This section aims to point out some interesting examples that can be obtained with our techniques; as these examples do not directly inform our main results, we omit some details.

\subsubsection{Exotic contractible Stein 4-manifolds}

First, we produce pairs of exotic contractible Stein manifolds, which were not known to exist in \cite{akbulut-ruberman}. Let $K \cup C$ be a link satisfying the hypotheses preceding Theorem~\ref{thm:reduce}, and let $k \subset S^1 \times S^2$ denote the knot induced by $K$ in $S^3_0(C)=S^1\times S^2$. The knot $k \subset S^1 \times S^2$ has a framing induced by the Seifert framing for $K \subset S^3$ and, as discussed in the beginning of \S\ref{subsec:construction}, this induces framings on the satellite knots $P(k),Q(k) \subset S^1 \times S^2$.

Viewing $S^1 \times S^2$ as the contact boundary of the Stein domain $S^1 \times B^3$, let $\overline{tb}(k) \in \zee$ denote the maximal contact framing achieved by any Legendrian representative of $k$ (as measured against the given framing on $k$). As discussed in the proof of \cite[Theorem~4.1]{yasui:conc}, if $\overline{tb}(k) \geq 1$, then $\overline{tb}(P(k)),\overline{tb}(Q(k)) \geq 1$. In this case, it follows that the Mazur manifolds $W$ and $W'$ obtained by attaching 2-handles to $S^1 \times B^3$ along the framed knots $P(k)$ and $Q(k)$ both admit Stein structures.  By taking examples with $\nu(K)=\tau(K)>0$, we obtain pairs of exotic Mazur manifolds that admit Stein structures.  

\begin{rem}
While this paper was in preparation, Akbulut and Yildiz \cite{akbulut-yildiz} produced examples of exotic contractible Stein manifolds using constructive methods from \cite{akbulut-ruberman}. Their examples, like those of \cite{akbulut-ruberman}, have more complicated handle structures. 
\end{rem}

\subsubsection{Exotic Mazur manifolds with unknotted 2-handles} Next, we outline a modification to the proof of Theorem~\ref{thm:mazur} that allows the hypothesis $\tau(K)=\nu(K)>0$ (of Section \ref{subsec:obstruction}) to be weakened. In particular, this allows us to produce exotic pairs of Mazur manifolds whose handlebody diagrams have unknotted 2-handles. Such Mazur manifolds are of special interest since their boundaries naturally admit two Mazur fillings; either zero-surgery curve can serve as the (dotted) one-handle. The operation of interchanging these two fillings, called \emph{dot-zero surgery}, is  related to cork twisting and is of central interest in the exotica literature. When one of the manifolds $W$ from this section is embedded in a 4-manifold $X$, it may be of interest to compare the manifolds obtained from $X$ by either by dot-zero surgery on $W$ or by replacing $W$ with our (absolutely) exotic filling $W'$. See Figure~\ref{fig:unknotted} for an example.

\begin{figure}\center
\def\svgwidth{\linewidth}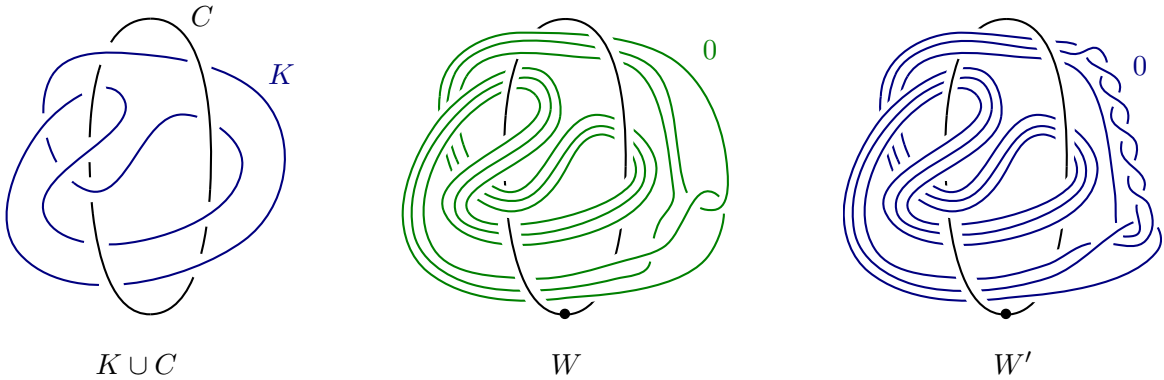 
\caption{A link $K \cup C$ that gives rise to an exotic pair of Mazur manifolds with unknotted 2-handles.}
\label{fig:unknotted}
\end{figure}

Let $L=K \cup C$ be a link satisfying the hypotheses preceding Theorem~\ref{thm:reduce}. Let us recall the notation and strategy from the proof of Theorem~\ref{thm:reduce}: Any diffeomorphism $\Phi$ between the associated Mazur manifolds $W_L$ and $W_L'$ restricts to a diffeomorphism $\phi$ between their boundaries that carries a certain 0-framed curve $\gamma \subset \partial W_L$ to a framed curve $\gamma'=\phi(\gamma) \subset \partial W_L'$. Moreover, if we attach 2-handles to $W_L$ and $W_L'$ along the framed knots $\gamma$ and $\gamma'$, respectively, we can extend the diffeomorphism between $W_L$ and $W_L'$ to a diffeomorphism between the knot traces $X_0(P(K))$ and $X_0(Q(K))$.  To obtain these knot traces, one repeatedly slides the 2-handle attached along $P(K)$ or $Q(K)$ over the 2-handle attached along $\gamma$ or $\gamma'$, respectively, until the former 2-handle no longer runs over the 1-handle. Then the 1-handle is canceled with the 2-handle attached along $\gamma$ or $\gamma'$, respectively.

Since the diffeomorphism $\phi$ carries the framed knot $\gamma$ to the framed knot $\gamma'$, we may assume that it carries a (framed) tubular neighborhood $N$ of $\gamma$ to a (framed) tubular neighborhood $N'$ of $\gamma'$. Thus, if we tie a local knot $J$ into $\gamma$ inside $N$, it has the effect of locally tying $J$ into $\gamma'$ inside $N'$. By keeping track of these tubular neighborhoods $N$ and $N'$, the proof of Theorem~\ref{thm:reduce} now shows that the 1-handles in $W_L$ and $W_L'$ can still be canceled with these (locally-knotted) 2-handles attached along $\gamma$ and $\gamma'$. We argued in the proof of Theorem \ref{thm:reduce} that sliding the 2-handles attached along $P(K)$ and $Q(K)$ over the 2-handles attached along $\gamma$ and $\gamma'$ returned knots $P(K)$ and $Q(K)$, respectively. Now that a local knot has been tied in $\gamma$ and $\gamma'$, the same argument can be modified to show slides instead yield knots $P(\widetilde{K}(J))$ and $Q(\widetilde{K}(J))$, where $\widetilde{K}$ is the pattern induced by considering $K$ in the complement of $C$.

With this modification of Theorem~\ref{thm:reduce} in hand, let $L=K \cup C$ denote the link from Figure~\ref{fig:unknotted} and $W_L$ and $W_L'$ the associated Mazur manifolds. In this case, it is easy to see that $K$ induces a pattern $\widetilde{K}$ which is concordant in $S^1\times D^2\times I$ to the core of the solid torus. Hence $\widetilde{K}(J)$ is concordant to $J$ for any knot $J$. As above, for any knot $J$ and any diffeomorphism between $W_L$ and $W_L'$, we obtain a diffeomorphism between the knot traces $X_0(P(\widetilde{K}(J)))$ and $X_0(Q(\widetilde{K}(J)))$. To be concrete, let $J$ be the the right-handed trefoil. Then 
$$\tau(\widetilde{K}(J))=\tau(J)=1=\nu(J)=\nu(\widetilde{K}(J)),$$
so Theorem~\ref{thm:PQtraces} implies that the traces $X_0(P(\widetilde{K}(J)))$ and $X_0(Q(\widetilde{K}(J)))$ cannot be diffeomorphic. It follows that $W_L$ and $W_L'$ are homeomorphic but not diffeomorphic.

 \subsubsection{Exotic Mazur manifolds with hyperbolic boundary} \label{sec:hyperbolic}

Let $W_n$ and $W_n'$ be the 4-manifolds shown in Figure~\ref{fig:hyperbolic}. Note that both are Mazur manifolds since the blue 2-handle links one of the 1-handles geometrically once, and hence the pair can be canceled. A homeomorphism between $W_n$ and $W_n'$ can be constructed in a manner analogous to the proof of Proposition~\ref{prop:homeo}, taking care to track the blue band that is linked with the green 2-handle.   
We claim that $Y_n=\partial W_n = \partial W_n'$ is a hyperbolic 3-manifold with no nontrivial self-diffeomorphisms. To see this, note that $Y_n$ can be obtained from $Y=Y_0$ by performing $-1/n$-surgery on the knot $\alpha \subset Y$ indicated by the dashed circle in Figure~\ref{fig:hyperbolic}. Using SnapPy and Sage \cite{snappy,sagemath}, we can verify that $Y \setminus \alpha$ is hyperbolic with no nontrivial self-diffeomorphisms. Therefore the same is true for $Y_n=Y_{-1/n}(\alpha)$ for $|n| \gg 0$; see, for example, \cite[Lemma~2.2]{dhl}.

\begin{figure}\center
\def\svgwidth{\linewidth}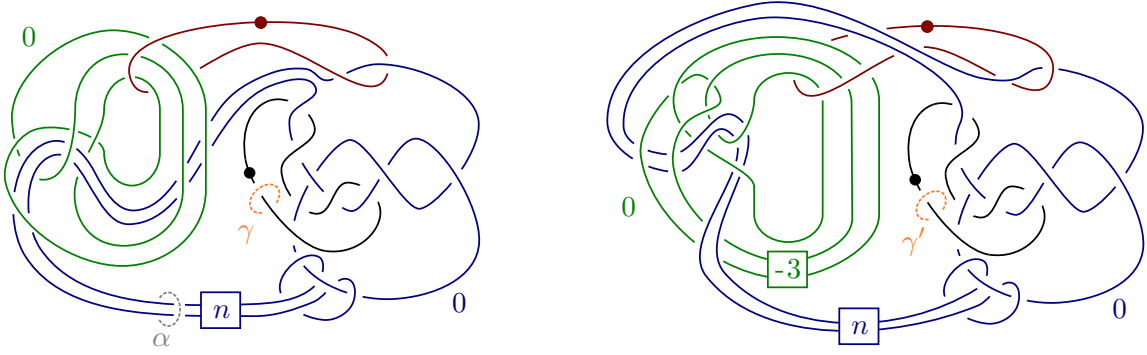 
\caption{A pair of homeomorphic Mazur manifolds with hyperbolic boundary (for $|n| \gg 0$).}
\label{fig:hyperbolic}
\end{figure}

By inspection, the given homeomorphism between $W_n$ and $W_n'$ is seen to carry $\gamma$ to $\gamma'$, preserving 0-framings. For $|n| \gg 0$, any diffeomorphism $W_n \cong W_n'$ must do the same (since there are no nontrivial self-diffeomorphisms of $Y_n$). Suppose such a diffeomorphism exists. Attaching 0-framed 2-handles along $\gamma$ and $\gamma'$ yields a diffeomorphism between the knot traces $X_0(K_n)$ and $X_0(K_n')$ described by erasing the black 1-handles from Figure~\ref{fig:hyperbolic} and canceling the blue 2-handles with the remaining 1-handles. It is straightforward to see that these knots $K_n$ and $K_n'$ are concordant to $P(T_{2,3})$ and $Q(T_{2,3})$, respectively, where $T_{2,3}$ denotes the right-handed trefoil. This implies that $\nu(K_n)=\nu(P(T_{2,3}))=2$ and $\nu(K_n')=\nu(Q(T_{2,3}))=1$, hence $X_0(K_n')$ and $X_0(K_n)$ cannot be diffeomorphic, a contradiction. It follows that $W_n$ and $W_n'$ cannot be diffeomorphic.

\begin{rem}
For $|n|\gg 0$, a similar argument shows that the knots $K_n$ and $K_n'$ described above are hyperbolic. To the authors' knowledge, these provide the first examples of hyperbolic knots with exotic zero-traces.
\end{rem}

\subsubsection{Exotic knot traces without genus bounds} \label{subsubsec:traces}
The first examples of exotic $n$-framed knot traces were given by Akbulut in \cite{akbulut:exotic}.  To the authors' knowledge, these and all other examples of exotic knot traces in the literature can be distinguished by the shake genera of the knots $K$ and $K'$ via the (shake-)slice Bennequin inequality. Similar results may be obtained using $\nu$ by applying Theorem~\ref{thm:nushakeboundclean} and Corollary~\ref{cor:shsliceben}.  
However, the invariant $\nu$  has the ability to distinguish between knot traces even when it fails to provide effective information about the shake genera of the knots. We illustrate this by constructing, for any $n>0$, a pair of exotic knot traces $X_0(K),X_0(K')$ such that $\nu(K)=2$, $\nu(K')=1$, yet $\gsh^0(K),\gsh^0(K') \geq 2n$.

Begin by setting $K_0=T_{2,5}\# 2T_{2,3} \# -T_{2,3;2,5}$, where $T_{2,3;2,5}$ denotes the $(2,5)$-cable of $T_{2,3}$, and then define $K=P(P(\#^n K_0))$ and $K'=Q(P(\#^n K_0))$. To calculate $\nu$, we work from the inside out, beginning with $K_0$:  By \cite[Corollary~3.2]{hom-wu}, we have $\tau(K_0)=0$, $\nu(K_0)=1$, and $\sigma(K_0)=-4$. It follows that $\tau( \#^n K_0)=n \cdot \tau(K_0)=0$ and $\sigma(\#^n K_0)=n\cdot \sigma(K_0)=-4n$. Since $\nu(K_0)=1$ and $\tau(K_0)=0$, we have $\epsilon(K_0) =-1$, implying $\epsilon(\#^n K_0)=-1$ (since $\epsilon$ is sign-additive under connected sum).  Since $\epsilon(\#^n K_0) = -1$, we have $\tau(P(\#^n K_0))=\tau(\#^n K_0)+1=1$ and $\epsilon(P(\#^n K_0))=1$. This gives $\nu(P(\#^n K_0))=\tau(P(\#^n K_0))=1
$.

Now since $\tau(P(\#^n K_0))>0$ and $\epsilon(P(\#^n K_0))=1$, we have
\begin{align*}
\nu(K)&= \nu(P(P(\#^n K_0)))=\nu(P(\#^n K_0))+1=2 \\
\nu(K')&= \nu(Q(P(\#^n K_0)))=\nu(P(\#^n K_0))=1.
\end{align*}
By Theorem~\ref{thm:PQtraces}, $X_0(K)$ and $X_0(K')$ are homeomorphic but not diffeomorphic.  

However, the bounds on $\gsh^0(K)$ and $\gsh^0(K')$ provided by $\nu$ (as described in Theorem~\ref{thm:nushakeboundclean}) are not sharp. To see this, we recall that the classical knot signature behaves nicely with respect to the satellite operation, see \cite{shinohara, litherland}.  In particular, for $J$ any winding number 1 satellite operator, we have
\begin{equation*}
\sigma(J(K))=\sigma(J(U))+\sigma(K)
\end{equation*}
Therefore we have
$$\sigma(K)=\sigma(K')=\sigma(\#^n K_0)=-4n.$$
Applying the bound  $\gsh^0 \geq|\sigma|/2$, a proof of which is sketched in \cite{akbulut2dclasses} (based on  \cite{tristram}), we conclude $\gsh^0(K),\gsh^0(K') \geq 2n$.

\bibliographystyle{gtart}
\bibliography{mazur-biblio}

\end{document}

%% file: example1.pdf_tex
\begingroup%
  \makeatletter%
  \providecommand\color[2][]{%
    \errmessage{(Inkscape) Color is used for the text in Inkscape, but the package 'color.sty' is not loaded}%
    \renewcommand\color[2][]{}%
  }%
  \providecommand\transparent[1]{%
    \errmessage{(Inkscape) Transparency is used (non-zero) for the text in Inkscape, but the package 'transparent.sty' is not loaded}%
    \renewcommand\transparent[1]{}%
  }%
  \providecommand\rotatebox[2]{#2}%
  \ifx\svgwidth\undefined%
    \setlength{\unitlength}{774.41857045bp}%
    \ifx\svgscale\undefined%
      \relax%
    \else%
      \setlength{\unitlength}{\unitlength * \real{\svgscale}}%
    \fi%
  \else%
    \setlength{\unitlength}{\svgwidth}%
  \fi%
  \global\let\svgwidth\undefined%
  \global\let\svgscale\undefined%
  \makeatother%
  \begin{picture}(1,0.41012764)%
    \put(0,0){\includegraphics[width=\unitlength,page=1]{example1.pdf}}%
    \put(0.17197007,0.55297251){\color[rgb]{0,0,0}\makebox(0,0)[lt]{\begin{minipage}{0.1586152\unitlength}\raggedright  \end{minipage}}}%
    \put(0,0){\includegraphics[width=\unitlength,page=2]{example1.pdf}}%
    \put(0.86731863,0.35752063){\color[rgb]{0,0,0.50196078}\makebox(0,0)[lb]{\smash{$-3$}}}%
    \put(0.38631615,0.12368361){\color[rgb]{0,0.50196078,0}\makebox(0,0)[lb]{\smash{$0$}}}%
    \put(0.19085363,0.02293095){\color[rgb]{0,0.50196078,0}\makebox(0,0)[lb]{\smash{$-3$}}}%
    \put(1.0194686,0.12809663){\color[rgb]{0,0,0.50196078}\makebox(0,0)[lb]{\smash{$0$}}}%
    \put(0.82400608,0.02734397){\color[rgb]{0,0,0.50196078}\makebox(0,0)[lb]{\smash{$-3$}}}%
    \put(0,0){\includegraphics[width=\unitlength,page=3]{example1.pdf}}%
  \end{picture}%
\endgroup%

%% file: patterns.pdf_tex
\begingroup%
  \makeatletter%
  \providecommand\color[2][]{%
    \errmessage{(Inkscape) Color is used for the text in Inkscape, but the package 'color.sty' is not loaded}%
    \renewcommand\color[2][]{}%
  }%
  \providecommand\transparent[1]{%
    \errmessage{(Inkscape) Transparency is used (non-zero) for the text in Inkscape, but the package 'transparent.sty' is not loaded}%
    \renewcommand\transparent[1]{}%
  }%
  \providecommand\rotatebox[2]{#2}%
  \ifx\svgwidth\undefined%
    \setlength{\unitlength}{329.83307846bp}%
    \ifx\svgscale\undefined%
      \relax%
    \else%
      \setlength{\unitlength}{\unitlength * \real{\svgscale}}%
    \fi%
  \else%
    \setlength{\unitlength}{\svgwidth}%
  \fi%
  \global\let\svgwidth\undefined%
  \global\let\svgscale\undefined%
  \makeatother%
  \begin{picture}(1,0.31681875)%
    \put(0,0){\includegraphics[width=\unitlength,page=1]{patterns.pdf}}%
    \put(0.21403049,0.00501304){\color[rgb]{0,0,0}\makebox(0,0)[lb]{\smash{$P$}}}%
    \put(0.76886214,0.00501298){\color[rgb]{0,0,0}\makebox(0,0)[lb]{\smash{$Q$}}}%
    \put(0,0){\includegraphics[width=\unitlength,page=2]{patterns.pdf}}%
    \put(0.69357,0.2475339){\color[rgb]{0,0,0.50196078}\makebox(0,0)[lb]{\smash{$-3$}}}%
  \end{picture}%
\endgroup%

%% file: homeo-general.pdf_tex
\begingroup%
  \makeatletter%
  \providecommand\color[2][]{%
    \errmessage{(Inkscape) Color is used for the text in Inkscape, but the package 'color.sty' is not loaded}%
    \renewcommand\color[2][]{}%
  }%
  \providecommand\transparent[1]{%
    \errmessage{(Inkscape) Transparency is used (non-zero) for the text in Inkscape, but the package 'transparent.sty' is not loaded}%
    \renewcommand\transparent[1]{}%
  }%
  \providecommand\rotatebox[2]{#2}%
  \ifx\svgwidth\undefined%
    \setlength{\unitlength}{677.30626384bp}%
    \ifx\svgscale\undefined%
      \relax%
    \else%
      \setlength{\unitlength}{\unitlength * \real{\svgscale}}%
    \fi%
  \else%
    \setlength{\unitlength}{\svgwidth}%
  \fi%
  \global\let\svgwidth\undefined%
  \global\let\svgscale\undefined%
  \makeatother%
  \begin{picture}(1,1.60809562)%
    \put(0,0){\includegraphics[width=\unitlength,page=1]{homeo-general.pdf}}%
    \put(0.86648269,1.10385041){\color[rgb]{0,0.50196078,0}\makebox(0,0)[lb]{\smash{$0$}}}%
    \put(0.65587977,1.10637125){\color[rgb]{0,0,0.50196078}\makebox(0,0)[lb]{\smash{$n$}}}%
    \put(0.83685694,0.70960752){\color[rgb]{0,0.50196078,0}\makebox(0,0)[lb]{\smash{$n-2$}}}%
    \put(0.65818386,0.68047017){\color[rgb]{0,0,0.50196078}\makebox(0,0)[lb]{\smash{$n$}}}%
    \put(0.05009169,0.32945198){\color[rgb]{0,0,0.50196078}\makebox(0,0)[lb]{\smash{$n$}}}%
    \put(0.08355061,1.57466833){\color[rgb]{0,0.50196078,0}\makebox(0,0)[lb]{\smash{$n$}}}%
    \put(0.63373338,1.57144921){\color[rgb]{0,0,0.50196078}\makebox(0,0)[lb]{\smash{$n$}}}%
    \put(0.96269371,1.55901007){\color[rgb]{0,0.50196078,0}\rotatebox{-180}{\makebox(0,0)[lb]{\smash{$0$}}}}%
    \put(0.15799637,1.24696648){\color[rgb]{0,0,0}\makebox(0,0)[lb]{\smash{(a)}}}%
    \put(0.80542878,1.24695205){\color[rgb]{0,0,0}\makebox(0,0)[lb]{\smash{(b)}}}%
    \put(0.15891903,0.80145761){\color[rgb]{0,0,0}\makebox(0,0)[lb]{\smash{(c)}}}%
    \put(0.80542878,0.80144331){\color[rgb]{0,0,0}\makebox(0,0)[lb]{\smash{(d)}}}%
    \put(0.80961038,0.41886362){\color[rgb]{0,0,0}\makebox(0,0)[lb]{\smash{(f)}}}%
    \put(0.1579675,0.41887805){\color[rgb]{0,0,0}\makebox(0,0)[lb]{\smash{(e)}}}%
    \put(0.15767911,0.00614218){\color[rgb]{0,0,0}\makebox(0,0)[lb]{\smash{(g)}}}%
    \put(0.80542878,0.00500316){\color[rgb]{0,0,0}\makebox(0,0)[lb]{\smash{(h)}}}%
    \put(0.05269593,0.98899405){\color[rgb]{0,0.50196078,0}\rotatebox{-180}{\makebox(0,0)[lb]{\smash{$0$}}}}%
    \put(0.16349222,1.35992989){\color[rgb]{0,0.50196078,0}\makebox(0,0)[lb]{\smash{$n$}}}%
    \put(-0.01009857,1.0984621){\color[rgb]{0,0,0.50196078}\makebox(0,0)[lb]{\smash{$n$}}}%
    \put(0.70783122,0.48935971){\color[rgb]{0,0,0}\makebox(0,0)[lb]{\smash{$n$}}}%
    \put(0.21790487,0.72735901){\color[rgb]{0,0.50196078,0}\makebox(0,0)[lb]{\smash{$n-2$}}}%
    \put(0.00730195,0.72987985){\color[rgb]{0,0,0.50196078}\makebox(0,0)[lb]{\smash{$n$}}}%
    \put(0.06130716,0.47609922){\color[rgb]{0,0,0}\rotatebox{-0.15702187}{\makebox(0,0)[lb]{\smash{$n$}}}}%
    \put(0.74464514,0.31963493){\color[rgb]{0,0,0.50196078}\makebox(0,0)[lb]{\smash{$n$}}}%
    \put(0.81952208,0.14711151){\color[rgb]{0,0,0.50196078}\makebox(0,0)[lb]{\smash{$-3$}}}%
    \put(0.81424258,0.0810301){\color[rgb]{0,0,0.50196078}\makebox(0,0)[lb]{\smash{$n$}}}%
    \put(0,0){\includegraphics[width=\unitlength,page=2]{homeo-general.pdf}}%
    \put(0.11083574,0.13160868){\color[rgb]{0,0,0.50196078}\makebox(0,0)[lb]{\smash{$-2$}}}%
    \put(0.10555599,0.06995666){\color[rgb]{0,0,0.50196078}\makebox(0,0)[lb]{\smash{$n$}}}%
    \put(0,0){\includegraphics[width=\unitlength,page=3]{homeo-general.pdf}}%
  \end{picture}%
\endgroup%

%% file: tref-phone-curve.pdf_tex
\begingroup%
  \makeatletter%
  \providecommand\color[2][]{%
    \errmessage{(Inkscape) Color is used for the text in Inkscape, but the package 'color.sty' is not loaded}%
    \renewcommand\color[2][]{}%
  }%
  \providecommand\transparent[1]{%
    \errmessage{(Inkscape) Transparency is used (non-zero) for the text in Inkscape, but the package 'transparent.sty' is not loaded}%
    \renewcommand\transparent[1]{}%
  }%
  \providecommand\rotatebox[2]{#2}%
  \ifx\svgwidth\undefined%
    \setlength{\unitlength}{338.3500347bp}%
    \ifx\svgscale\undefined%
      \relax%
    \else%
      \setlength{\unitlength}{\unitlength * \real{\svgscale}}%
    \fi%
  \else%
    \setlength{\unitlength}{\svgwidth}%
  \fi%
  \global\let\svgwidth\undefined%
  \global\let\svgscale\undefined%
  \makeatother%
  \begin{picture}(1,0.78580792)%
    \put(0,0){\includegraphics[width=\unitlength,page=1]{tref-phone-curve.pdf}}%
    \put(0.84114381,0.56094358){\color[rgb]{0,0,0.50196078}\makebox(0,0)[lb]{\smash{$K$}}}%
    \put(-0.10529751,0.56094358){\color[rgb]{0,0,0}\makebox(0,0)[lb]{\smash{$C$}}}%
    \put(0.8063111,0.3456128){\color[rgb]{0.4,0.4,0.4}\makebox(0,0)[lb]{\smash{$\alpha$}}}%
  \end{picture}%
\endgroup%

%% file: reduce.pdf_tex
\begingroup%
  \makeatletter%
  \providecommand\color[2][]{%
    \errmessage{(Inkscape) Color is used for the text in Inkscape, but the package 'color.sty' is not loaded}%
    \renewcommand\color[2][]{}%
  }%
  \providecommand\transparent[1]{%
    \errmessage{(Inkscape) Transparency is used (non-zero) for the text in Inkscape, but the package 'transparent.sty' is not loaded}%
    \renewcommand\transparent[1]{}%
  }%
  \providecommand\rotatebox[2]{#2}%
  \ifx\svgwidth\undefined%
    \setlength{\unitlength}{638.96147bp}%
    \ifx\svgscale\undefined%
      \relax%
    \else%
      \setlength{\unitlength}{\unitlength * \real{\svgscale}}%
    \fi%
  \else%
    \setlength{\unitlength}{\svgwidth}%
  \fi%
  \global\let\svgwidth\undefined%
  \global\let\svgscale\undefined%
  \makeatother%
  \begin{picture}(1,0.33141103)%
    \put(0,0){\includegraphics[width=\unitlength,page=1]{reduce.pdf}}%
    \put(0.27746793,0.08796601){\color[rgb]{0,0.50196078,0}\makebox(0,0)[lb]{\smash{$P_n$}}}%
    \put(0,0){\includegraphics[width=\unitlength,page=2]{reduce.pdf}}%
    \put(0.33983514,0.02223591){\color[rgb]{0,0.50196078,0}\makebox(0,0)[lb]{\smash{$n$}}}%
    \put(0.95020073,0.02223591){\color[rgb]{0,0.50196078,0}\makebox(0,0)[lb]{\smash{$n$}}}%
    \put(0,0){\includegraphics[width=\unitlength,page=3]{reduce.pdf}}%
    \put(0.49045739,0.17457448){\color[rgb]{0,0,0}\makebox(0,0)[lb]{\smash{$\Phi$}}}%
    \put(-0.00155892,0.25638433){\color[rgb]{0,0,0.50196078}\makebox(0,0)[lb]{\smash{$\gamma$}}}%
    \put(0,0){\includegraphics[width=\unitlength,page=4]{reduce.pdf}}%
    \put(0.60729623,0.30672868){\color[rgb]{0,0,0.50196078}\makebox(0,0)[lb]{\smash{$\phi(\gamma)$}}}%
    \put(0.34022854,0.2916007){\color[rgb]{0.50196078,0.50196078,0.50196078}\makebox(0,0)[lb]{\smash{$T$}}}%
    \put(0.95059423,0.2916007){\color[rgb]{0.50196078,0.50196078,0.50196078}\makebox(0,0)[lb]{\smash{$T'$}}}%
    \put(0,0){\includegraphics[width=\unitlength,page=5]{reduce.pdf}}%
    \put(0.88446914,0.08796601){\color[rgb]{0,0.50196078,0}\makebox(0,0)[lb]{\smash{$Q_n$}}}%
    \put(0.65289916,0.13676337){\color[rgb]{0,0,0.50196078}\makebox(0,0)[lb]{\smash{$m$}}}%
    \put(0.03783869,0.14850101){\color[rgb]{0,0,0.50196078}\makebox(0,0)[lb]{\smash{$0$}}}%
    \put(0,0){\includegraphics[width=\unitlength,page=6]{reduce.pdf}}%
  \end{picture}%
\endgroup%

%% file: tauknots.pdf_tex
\begingroup%
  \makeatletter%
  \providecommand\color[2][]{%
    \errmessage{(Inkscape) Color is used for the text in Inkscape, but the package 'color.sty' is not loaded}%
    \renewcommand\color[2][]{}%
  }%
  \providecommand\transparent[1]{%
    \errmessage{(Inkscape) Transparency is used (non-zero) for the text in Inkscape, but the package 'transparent.sty' is not loaded}%
    \renewcommand\transparent[1]{}%
  }%
  \providecommand\rotatebox[2]{#2}%
  \ifx\svgwidth\undefined%
    \setlength{\unitlength}{9222.38378906bp}%
    \ifx\svgscale\undefined%
      \relax%
    \else%
      \setlength{\unitlength}{\unitlength * \real{\svgscale}}%
    \fi%
  \else%
    \setlength{\unitlength}{\svgwidth}%
  \fi%
  \global\let\svgwidth\undefined%
  \global\let\svgscale\undefined%
  \makeatother%
  \begin{picture}(1,0.30384195)%
    \put(0,0){\includegraphics[width=\unitlength,page=1]{tauknots.pdf}}%
    \put(0.12881074,0.00149524){\color[rgb]{0,0,0}\makebox(0,0)[lb]{\smash{$X$}}}%
    \put(0.48571468,0.00149524){\color[rgb]{0,0,0}\makebox(0,0)[lb]{\smash{$K$}}}%
    \put(0.84438732,0.00149524){\color[rgb]{0,0,0}\makebox(0,0)[lb]{\smash{$K'$}}}%
    \put(0.01774462,0.24005747){\color[rgb]{0,0,0.50196078}\makebox(0,0)[lb]{\smash{$0$}}}%
    \put(0.08477862,0.2683169){\color[rgb]{0,0.50196078,0}\makebox(0,0)[lb]{\smash{$0$}}}%
  \end{picture}%
\endgroup%

%% file: unknotted-ex.pdf_tex
\begingroup%
  \makeatletter%
  \providecommand\color[2][]{%
    \errmessage{(Inkscape) Color is used for the text in Inkscape, but the package 'color.sty' is not loaded}%
    \renewcommand\color[2][]{}%
  }%
  \providecommand\transparent[1]{%
    \errmessage{(Inkscape) Transparency is used (non-zero) for the text in Inkscape, but the package 'transparent.sty' is not loaded}%
    \renewcommand\transparent[1]{}%
  }%
  \providecommand\rotatebox[2]{#2}%
  \ifx\svgwidth\undefined%
    \setlength{\unitlength}{786.20369979bp}%
    \ifx\svgscale\undefined%
      \relax%
    \else%
      \setlength{\unitlength}{\unitlength * \real{\svgscale}}%
    \fi%
  \else%
    \setlength{\unitlength}{\svgwidth}%
  \fi%
  \global\let\svgwidth\undefined%
  \global\let\svgscale\undefined%
  \makeatother%
  \begin{picture}(1,0.3172053)%
    \put(0,0){\includegraphics[width=\unitlength,page=1]{unknotted-ex.pdf}}%
    \put(0.60114042,0.27742404){\color[rgb]{0,0.50196078,0}\makebox(0,0)[lb]{\smash{$0$}}}%
    \put(0.46958379,0.00529146){\color[rgb]{0,0,0}\makebox(0,0)[lb]{\smash{$W$}}}%
    \put(0,0){\includegraphics[width=\unitlength,page=2]{unknotted-ex.pdf}}%
    \put(0.97315885,0.2640686){\color[rgb]{0,0,0.50196078}\makebox(0,0)[lb]{\smash{$0$}}}%
    \put(0.85275996,0.00529157){\color[rgb]{0,0,0}\makebox(0,0)[lb]{\smash{$W'$}}}%
    \put(0,0){\includegraphics[width=\unitlength,page=3]{unknotted-ex.pdf}}%
    \put(0.15826345,0.30633361){\color[rgb]{0,0,0}\makebox(0,0)[lb]{\smash{$C$}}}%
    \put(0.22502698,0.25576176){\color[rgb]{0,0,0.50196078}\makebox(0,0)[lb]{\smash{$K$}}}%
    \put(0.07582032,0.00529146){\color[rgb]{0,0,0}\makebox(0,0)[lb]{\smash{$K \cup C$}}}%
  \end{picture}%
\endgroup%

%% file: hyperbolicMazur.pdf_tex
\begingroup%
  \makeatletter%
  \providecommand\color[2][]{%
    \errmessage{(Inkscape) Color is used for the text in Inkscape, but the package 'color.sty' is not loaded}%
    \renewcommand\color[2][]{}%
  }%
  \providecommand\transparent[1]{%
    \errmessage{(Inkscape) Transparency is used (non-zero) for the text in Inkscape, but the package 'transparent.sty' is not loaded}%
    \renewcommand\transparent[1]{}%
  }%
  \providecommand\rotatebox[2]{#2}%
  \ifx\svgwidth\undefined%
    \setlength{\unitlength}{1292.93692521bp}%
    \ifx\svgscale\undefined%
      \relax%
    \else%
      \setlength{\unitlength}{\unitlength * \real{\svgscale}}%
    \fi%
  \else%
    \setlength{\unitlength}{\svgwidth}%
  \fi%
  \global\let\svgwidth\undefined%
  \global\let\svgscale\undefined%
  \makeatother%
  \begin{picture}(1,0.29765037)%
    \put(1.07412366,0.11426778){\color[rgb]{0,0,0}\makebox(0,0)[lb]{\smash{}}}%
    \put(0,0){\includegraphics[width=\unitlength,page=1]{hyperbolicMazur.pdf}}%
    \put(0.0138321,0.25928646){\color[rgb]{0,0.50196078,0}\makebox(0,0)[lb]{\smash{$0$}}}%
    \put(0.38590402,0.02849964){\color[rgb]{0,0,0.50196078}\makebox(0,0)[lb]{\smash{$0$}}}%
    \put(0.53247784,0.11162355){\color[rgb]{0,0.50196078,0}\makebox(0,0)[lb]{\smash{$0$}}}%
    \put(0.95731021,0.02382371){\color[rgb]{0,0,0.50196078}\makebox(0,0)[lb]{\smash{$0$}}}%
    \put(0.66723019,0.05768002){\color[rgb]{0,0.50196078,0}\makebox(0,0)[lb]{\smash{-$3$}}}%
    \put(0.73136965,0.0096192){\color[rgb]{0,0,0.50196078}\makebox(0,0)[lb]{\smash{$n$}}}%
    \put(0.17914289,0.02354117){\color[rgb]{0,0,0.50196078}\makebox(0,0)[lb]{\smash{$n$}}}%
    \put(0.2007258,0.09276973){\color[rgb]{1,0.49803922,0.16470588}\makebox(0,0)[lb]{\smash{$\gamma$}}}%
    \put(0.12656047,-0.00111952){\color[rgb]{0.50196078,0.50196078,0.50196078}\makebox(0,0)[lb]{\smash{$\alpha$}}}%
    \put(0.77499936,0.08232852){\color[rgb]{1,0.49803922,0.16470588}\makebox(0,0)[lb]{\smash{$\gamma'$}}}%
    \put(0,0){\includegraphics[width=\unitlength,page=2]{hyperbolicMazur.pdf}}%
  \end{picture}%
\endgroup%